\documentclass[reqno,11pt,letterpaper]{amsart}

\usepackage{amsmath,amssymb,amsthm,mathrsfs,bbm}
\usepackage{accents}
\usepackage{arydshln}
\usepackage{upgreek}
\usepackage{slashed}
\usepackage{xifthen}
\usepackage{graphicx}
\usepackage{longtable}
\usepackage[inline]{enumitem}

\usepackage{xcolor}
\definecolor{winered}{rgb}{0.7,0,0}
\definecolor{lessblue}{rgb}{0,0,0.7}

\usepackage[pdftex,colorlinks=true,linkcolor=winered,citecolor=lessblue,breaklinks=true,bookmarksopen=true]{hyperref}

\makeatletter
\newcommand{\myitem}[3]{\item[#2]\def\@currentlabel{#3}\label{#1}}
\makeatother

\usepackage{titletoc}

\makeatletter

\def\@tocline#1#2#3#4#5#6#7{
\begingroup
  \par
    \parindent\z@ \leftskip#3 \relax \advance\leftskip\@tempdima\relax
                  \rightskip\@pnumwidth plus 4em \parfillskip-\@pnumwidth
    \ifcase #1 
       \vskip 0.6em \hskip 0em 
       \or
       \or \hskip 0em 
       \or \hskip 1em 
    \fi%
    #6
    \nobreak\relax{\leavevmode\leaders\hbox{\,.}\hfill}
    \hbox to\@pnumwidth {\@tocpagenum{#7}}
  \par
\endgroup
}

 \def\l@section{\@tocline{0}{0pt}{0pc}{}{}}

\renewcommand{\tocsection}[3]{%
  \indentlabel{\@ifnotempty{#2}{ 
    \ignorespaces\bfseries{#2. #3}}}
  \indentlabel{\@ifempty{#2}{\ignorespaces\bfseries{#3}}{}} 
    \vspace{1.5pt}}

\renewcommand{\tocsubsection}[3]{%
  \indentlabel{\@ifnotempty{#2}{
    \ignorespaces#2. #3}}
  \indentlabel{\@ifempty{#2}{\ignorespaces #3}{}}
    \vspace{1.5pt}}

\renewcommand{\tocsubsubsection}[3]{%
  \indentlabel{\@ifnotempty{#2}{
    \ignorespaces#2. #3}}
  \indentlabel{\@ifempty{#2}{\ignorespaces #3}{}}
    \vspace{1.5pt}}

\makeatother

\makeatletter
\def\@nomenstarted{0}
\newlength{\@nomenoldtabcolsep}

\newcommand{\nomenstart}
  {%
    \def\@nomenstarted{1}%
    \setlength{\@nomenoldtabcolsep}{\tabcolsep}%
    \setlength{\tabcolsep}{3.5pt}%
    \begin{longtable}{p{0.11\textwidth} p{0.86\textwidth}}
  }

\newcommand{\nomenitem}[2]{%
    \ifcase\@nomenstarted%
      \or 
      \or \\ 
    \fi%
    #1\,{\leavevmode\leaders\hbox{\,.}\hfill} & #2%
    \def\@nomenstarted{2}%
  }%
\newcommand{\nomenend}
  {\\%
      \end{longtable}%
      \setlength{\tabcolsep}{\@nomenoldtabcolsep}%
      \def\@nomenstarted{0}%
  }
\makeatother

\makeatletter
\newcommand{\vast}{\bBigg@{4}}
\newcommand{\Vast}{\bBigg@{5}}
\makeatother

\allowdisplaybreaks

\numberwithin{equation}{section}
\numberwithin{figure}{section}
\newtheorem{thm}{Theorem}[section]

\newtheorem{prop}[thm]{Proposition}
\newtheorem{lemma}[thm]{Lemma}
\newtheorem{cor}[thm]{Corollary}
\newtheorem*{thm*}{Theorem}
\newtheorem*{prop*}{Proposition}
\newtheorem*{cor*}{Corollary}
\newtheorem*{conj*}{Conjecture}

\theoremstyle{definition}
\newtheorem{definition}[thm]{Definition}

\theoremstyle{remark}
\newtheorem{rmk}[thm]{Remark}

\newcommand{\mc}{\mathcal}
\newcommand{\cA}{\mc A}

\newcommand{\cC}{\mc C}

\newcommand{\cH}{\mc H}

\newcommand{\cL}{\mc L}

\newcommand{\cN}{\mc N}
\newcommand{\cO}{\mc O}

\newcommand{\cS}{\mc S}

\newcommand{\cV}{\mc V}

\newcommand{\ms}{\mathscr}

\newcommand{\sC}{\ms C}
\newcommand{\sD}{\ms D}

\newcommand{\scri}{\ms I}

\newcommand{\sR}{\ms R}

\newcommand{\C}{\mathbb{C}}
\newcommand{\N}{\mathbb{N}}
\newcommand{\R}{\mathbb{R}}
\newcommand{\Z}{\mathbb{Z}}

\newcommand{\Sph}{\mathbb{S}}

\newcommand{\sfG}{\mathsf{G}}

\newcommand{\bfa}{\mathbf{a}}

\newcommand{\fa}{\mathfrak{a}}

\newcommand{\fm}{\mathfrak{m}}

\newcommand{\so}{\mathfrak{so}}

\newcommand{\Err}{{\mathrm{Err}}{}}

\newcommand{\vol}{\operatorname{vol}}

\newcommand{\End}{\operatorname{End}}

\newcommand{\Hom}{\operatorname{Hom}}

\newcommand{\mathspan}{\operatorname{span}}
\newcommand{\supp}{\operatorname{supp}}

\newcommand{\tr}{\operatorname{tr}}

\newcommand{\dS}{{\mathrm{dS}}}

\newcommand{\Ups}{\Upsilon}

\newcommand{\eps}{\epsilon}

\newcommand{\hra}{\hookrightarrow}
\newcommand{\la}{\langle}

\newcommand{\ol}{\overline}
\newcommand{\pa}{\partial}
\newcommand{\ra}{\rangle}

\newcommand{\bop}{{\mathrm{b}}}

\newcommand{\Diff}{\mathrm{Diff}}

\newcommand{\Diffb}{\Diff_\bop}

\newcommand{\half}{{\tfrac{1}{2}}}

\newcommand{\loc}{{\mathrm{loc}}}
\newcommand{\CI}{\cC^\infty}
\newcommand{\CIdot}{\dot\cC^\infty}

\newcommand{\Hext}{\bar H}

\newcommand{\Ric}{\mathrm{Ric}}

\newcommand{\bhm}{\fm}
\newcommand{\bha}{\bfa}

\newcommand{\openbigpmatrix}[1]
  {%
    \def\@bigpmatrixsize{#1}%
    \addtolength{\arraycolsep}{-#1}%
    \begin{pmatrix}%
  }
\newcommand{\closebigpmatrix}
  {%
    \end{pmatrix}%
    \addtolength{\arraycolsep}{\@bigpmatrixsize}%
  }

\newlength{\enummargin}

\newcommand{\usref}[1]{{\upshape\ref{#1}}}

\DeclareGraphicsExtensions{.mps}

\makeatletter
\newcommand*{\fwbw}[1]{\expandafter\@fwbw\csname c@#1\endcsname}
\newcommand*{\@fwbw}[1]{\ifcase #1 \or {\rm fw}\or {\rm bw}\fi}
\AddEnumerateCounter{\fwbw}{\@fwbw}
\makeatother

\begin{document}

\title{Black hole gluing in de Sitter space}

\date{\today}


\author{Peter Hintz}
\address{Department of Mathematics, Massachusetts Institute of Technology, Cambridge, Massachusetts 02139-4307, USA}
\email{phintz@mit.edu}

\begin{abstract}
  We construct dynamical many-black-hole spacetimes with well-controlled asymptotic behavior as solutions of the Einstein vacuum equation with positive cosmological constant. We accomplish this by gluing Schwarzschild--de~Sitter or Kerr--de~Sitter black hole metrics into neighborhoods of points on the future conformal boundary of de~Sitter space, under certain balance conditions on the black hole parameters. We give a self-contained treatment of solving the Einstein equation directly for the metric, given the scattering data we encounter at the future conformal boundary. The main step in the construction is the solution of a linear divergence equation for trace-free symmetric 2-tensors; this is closely related to Friedrich's analysis of scattering problems for the Einstein equation on asymptotically simple spacetimes.
\end{abstract}

\maketitle

\section{Introduction}
\label{SI}

A vacuum spacetime with cosmological constant $\Lambda\in\R$ is a 4-manifold $M$ equipped with a Lorentzian metric $g$ of signature $({-}{+}{+}{+})$ satisfying the Einstein vacuum equation
\begin{equation}
\label{EqIEin}
  \Ric(g) - \Lambda g = 0.
\end{equation}
The Majumdar--Papapetrou \cite{MajumdarSolution,PapapetrouSolution} spacetime is an explicit solution for the coupled Einstein--Maxwell system\footnote{This means that the right hand side of~\eqref{EqIEin} is no longer $0$, but related to the energy-momentum tensor of an electromagnetic field satisfying Maxwell's equation.} in $\Lambda=0$ describing several \emph{extremally charged} black holes; a similar construction for $\Lambda>0$ was given by Kastor and Traschen \cite{KastorTraschenManyBH}. We will demonstrate how to construct \emph{vacuum} spacetimes which, for late times, describe dynamical many-black-hole spacetimes \emph{with precisely controlled asymptotic structure} using a gluing method. Our construction applies in the case $\Lambda>0$, which is consistent with the $\Lambda$CDM model currently favored in cosmology \cite{RiessEtAlLambda,PerlmutterEtAlLambda}.

The simplest solution of~\eqref{EqIEin} is de~Sitter space
\[
  M^\circ = (-\pi/2,\pi/2)_s \times \Sph^3,\quad
  g_\dS = (3/\Lambda)\cos^{-2}(s)\bigl(-d s^2+g_{\Sph^3}),
\]
where $g_{\Sph^3}$ is the standard metric on the 3-sphere; this describes an exponentially expanding (as $s\to\pi/2$) universe. The metric $g_\dS$ is asymptotically simple \cite{PenroseAsymptotics}: the conformal multiple $\cos^2(s)g_\dS$ extends smoothly to a Lorentzian metric on the partial compactification
\[
  M = (-\pi/2,\pi/2]_s \times \Sph^3.
\]
$(M^\circ,g_\dS)$ is geodesically complete, so future timelike observers in $M^\circ$ can only tend to $\pa M$ but never reach it; one calls $\pa M$ \emph{future timelike infinity}, or the \emph{future conformal boundary} of de~Sitter space, often also denoted $I^+$. Since images of null-geodesics are conformally invariant, the backward light cone from a point $p\in\pa M$ is a null hypersurface inside $M^\circ$ and known as the \emph{cosmological horizon} associated with $p$. See Figure~\ref{FigIdS}.

\begin{figure}[!ht]
\centering
\includegraphics{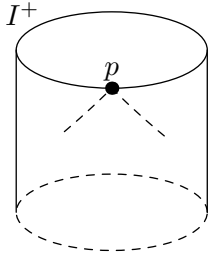}
\caption{The (partial) conformal compactification $M$ of de~Sitter space, a point $p$ on its future conformal boundary $I^+$, and a piece of the backwards light cone from $p$.}
\label{FigIdS}
\end{figure}

The simplest black hole solution of~\eqref{EqIEin} is the Schwarzschild--de~Sitter (SdS) solution, recalled below. It depends on a mass parameter $\bhm\in\R$ and can be thought of as describing a black hole tending to some fixed but arbitrary point $p$ in $I^+$; it is defined in a neighborhood of $p$. Our main result gives a sufficient condition under which one can glue several SdS black holes into de~Sitter space:

\begin{thm}
\label{ThmIBaby}
  Let $N\in\N$. For $i=1,\ldots,N$, fix points $p_i\in\pa M=\Sph^3\subset\R^4$ and (subextremal) masses $0<\bhm_i<(3\Lambda)^{-1/2}$ such that the balance condition
  \begin{equation}
  \label{EqIBabyBalance}
    \sum_{i=1}^N \bhm_i p_i=0 \in \R^4.
  \end{equation}
  holds. Then there exists a metric $g$ solving the Einstein vacuum equation~\eqref{EqIEin} in a neighborhood of $\pa M$ with the following properties:
  \begin{enumerate}
  \item\label{ItIBabySdS} in a neighborhood of $p_i$, $g$ is isometric to a Schwarzschild--de~Sitter black hole metric with mass $\bhm_i$, containing future affine complete event and cosmological horizons;
  \item\label{ItIBabydS} outside a small neighborhood of $\{p_1,\ldots,p_N\}$, $\cos^2(s) g$ is smooth down to $s=\pi/2$, and asymptotic to the rescaled de~Sitter metric $\cos^2(s)g_\dS$ at the rate $\cos^3(s)$.
  \end{enumerate}
\end{thm}

See Figure~\ref{FigIBaby}. When $N\geq 2$, and all masses are sufficiently small in absolute value, we show that the cosmological horizons of different black holes intersect in the maximal globally hyperbolic development of $g$; see the end of~\S\ref{SsSC}. Note that upon replacing $s$ by $-s$, we glue SdS black holes, with \emph{past} affine complete horizons, into a neighborhood of past conformal infinity of de~Sitter space; this provides interesting settings in which to (numerically) study the interaction of black holes in de~Sitter space under forward evolution.

\begin{figure}[!ht]
\centering
\includegraphics{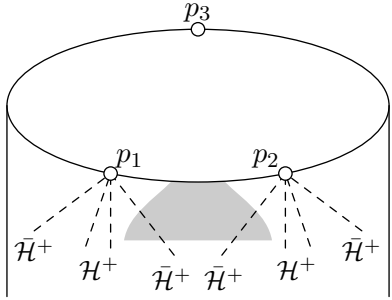}
\caption{Illustration of Theorem~\ref{ThmIBaby}. We glue SdS black holes into neighborhoods of the points $p_i$; only two black holes are shown here. The dashed lines labelled $\bar\cH^+$ are the cosmological horizons of the individual black holes, while the dashed lines labelled $\cH^+$ (not drawn Penrose-diagrammatically) indicate their event horizons. (The two $\bar\cH^+$ lines tending to $p_1$ are really a single $(0,\infty)_{t_*}\times\Sph^2$, forming one connected horizon, but for visualization purposes we needed to reduce dimension of the sphere by $2$.) The gray region indicates the region where the metric is not isometric to some SdS metric.}
\label{FigIBaby}
\end{figure}

Recall here that for subextremal mass parameters $\bhm\in(0,(3\Lambda)^{-1/2})$, the SdS metric is
\[
  \R_t \times (r_-,r_+)_r \times \Sph^2, \quad
  g_\bhm = -\mu_\bhm(r)d t^2 + \mu_\bhm(r)^{-1}\,d r^2 + r^2 g_{\Sph^2},
\]
where $\mu_\bhm(r)=1-\frac{2\bhm}{r}-\frac{\Lambda r^2}{3}$, and $0<r_-<r_+$ are the unique positive real roots of $\mu_\bhm$. After a suitable coordinate change, one can extend $g_\bhm$ beyond the event horizon $r=r_-$ and beyond the cosmological horizon $r=r_+$ to a metric $g_\bhm$ on a larger manifold
\[
  M_\bhm^\circ = \R_{t_*} \times (0,\infty)_r \times \Sph^2.
\]
One can identify the piece $t_*>0$ of $M_\bhm^\circ$ with a subset of de~Sitter space $M^\circ$ in such a way that the SdS cosmological horizon and the backward light cone from a point $p\in\pa M$ coincide in a neighborhood of $p$; denote the resulting metric by $g_{p,\bhm}$. This metric is in fact conformally smooth down to $\pa M$ away from the singular point $p$, with $r\to\infty$ corresponding to $s\to\pi/2$. See Figure~\ref{FigISdS}. Conclusion~\eqref{ItIBabySdS} in Theorem~\ref{ThmIBaby} is then the statement that $g=g_{p_i,\bhm_i}$ near $p_i$.

\begin{figure}[!ht]
\centering
\includegraphics{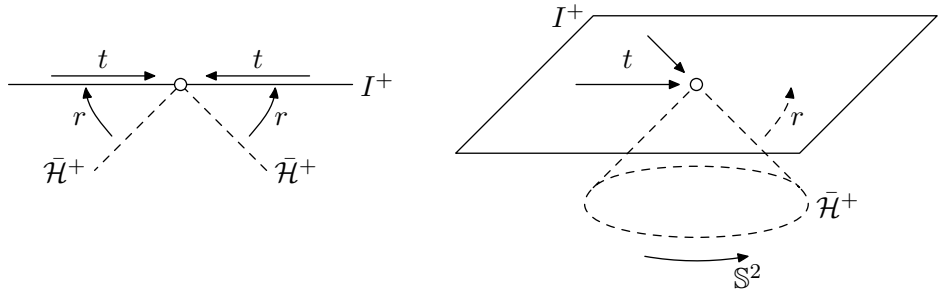}
\caption{The Schwarzschild--de~Sitter metric glued into de~Sitter space. We only show the cosmological horizon and the cosmological region $r>r_+$ where $r$ is timelike. On the right is the same picture, but we show an additional spatial dimension, thus showcasing the connectedness of $\bar\cH^+$.}
\label{FigISdS}
\end{figure}

The precise result, Theorem~\ref{ThmSGlue}, is more general: the masses $\bhm_i$ are allowed to be any real number, and we then glue the far end $r\gg 1$ of $M_\bhm^\circ$ into de~Sitter space. (For subextremal masses as in Theorem~\ref{ThmIBaby}, one can then `fill in' the rest of the SdS black hole.) We also prove the necessity of the balance condition under certain asymptotic assumptions on $g$; see Theorem~\ref{ThmSUniq}.

We prove a similar result for gluing Kerr--de~Sitter (KdS) black holes into de~Sitter space. The KdS family of metrics \cite{CarterHamiltonJacobiEinstein} depends on two parameters, $\bhm$ (mass) and $\bha$ (specific angular momentum). For the purpose of gluing KdS metrics into de~Sitter space, we also keep track of the point on $I^+$ to which the black hole tends, and the orientation of its axis of rotation. We can then glue any finite number of KdS black holes into de~Sitter space under two balance conditions: the first condition is similar to~\eqref{EqIBabyBalance} but now involving the effective mass $\bhm_{\rm eff}=\bhm/(1+\Lambda\bha^2/3)^2$, and the second condition requires the effective angular momenta $\bhm_{\rm eff}\bha$ (taking into account the black hole locations and axes of rotation) to sum up to $0$, see Definition~\ref{DefKBalance}.

\subsection{Gluing in general relativity}

Most gluing constructions in general relativity operate on the level of (noncharacteristic) initial data sets. Recall that an initial data set for the Einstein vacuum equation~\eqref{EqIEin} is a 3-manifold $\Sigma$ together with a Riemannian metric $\gamma$ and a symmetric 2-tensor $k$ on $\Sigma$ satisfying the constraint equations
\begin{equation}
\label{EqIConstr}
  R_\gamma + (\tr_\gamma k)^2 - |k|_\gamma^2 = 2\Lambda,\qquad
  \delta_\gamma k+d\tr_\gamma k = 0;
\end{equation}
here $R_\gamma$ is the scalar curvature, and $\delta_\gamma$ is the negative divergence (the adjoint of the symmetric gradient). Given $(\Sigma,\gamma,k)$, there exists a unique maximal globally hyperbolic development $(M,g)$, with $g$ solving~\eqref{EqIEin}, and an embedding $\Sigma\hra M$ such that the images of $\gamma$ and $k$ are the metric and second fundamental form of $\Sigma$ \cite{ChoquetBruhatLocalEinstein,ChoquetBruhatGerochMGHD}.

Brill--Lindquist \cite{BrillLindquist} explicitly constructed initial data containing any number $N$ of (charged) Einstein--Rosen bridges at arbitrarily chosen points in $\R^3$ and with arbitrary mass parameters; the resulting initial data has one distinguished asymptotically flat (AF) region and $N$ AF regions on the other side of the wormholes. Misner \cite{MisnerGeometrostatics} (and Lindquist \cite{LindquistIVP} in the Einstein--Maxwell case) showed how with a careful choice of parameters, one can identify all but two AF regions, and for just two points even create a spacetime with one AF end and a wormhole connecting two `points'. These constructions are global and rigid, the main tool being superpositions of shifted and scaled versions of $1/|x|$; this is also the case for Majumdar--Papapetrou and Kastor--Traschen spacetimes.

The starting point for localized gluing is Corvino's work \cite{CorvinoScalar} on the gluing of the large end of Schwarzschild data ($\Lambda=0$) to a given time-symmetric AF initial data set on $\R^3$; in this case, the constraint equations become $k=0$ and $R_\gamma=0$, and key to the localized gluing is the underdetermined nature of the scalar curvature operator $\gamma\mapsto R_\gamma$ (more precisely, the overdetermined nature of the adjoint of its linearization). The assumption of time-symmetry was removed by Corvino--Schoen \cite{CorvinoSchoenAsymptotics} by allowing the AF end to be equal to Kerr initial data.

Chru\'sciel--Delay \cite{ChruscielDelaySimple} extended the methods of Corvino--Schoen and also refined wormhole constructions by Isenberg--Mazzeo--Pollack \cite{IsenbergMazzeoPollackWormholes,IsenbergMazzeoPollackTopology}. In \cite[\S4]{ChruscielDelaySimple}, they constructed time-symmetric data containing any number of Schwarzschild black holes (meaning: neighborhoods of the neck region of the Riemannian Schwarzschild metric), placed at a collection of points in $\R^3$ which is symmetric around $0\in\R^3$. (This assumption is loosely related to the balance condition~\eqref{EqIBabyBalance}.) The same authors also construct many-Kerr initial data \cite[\S8.9]{ChruscielDelayMapping}, again under a parity condition. In both papers, the smallness required for solving the nonlinear constraint equations comes from taking the black hole masses to be small compared to the distance of the points. Chru\'sciel--Mazzeo \cite{ChruscielMazzeoManyBH} show that the maximally globally hyperbolic development of suitable many-Schwarzschild initial data has past-complete $\scri^+$, and the black hole region has several connected components. Their arguments use Friedrich's stability result \cite{FriedrichStability} and direct geometric arguments, a description of the \emph{global} structure of the resulting spacetime being far beyond the reach of hyperbolic PDE theory at this point.

Chru\'sciel--Isenberg--Pollack \cite{ChruscielIsenbergPollackEngineering,ChruscielIsenbergPollackPRL} give sharp results on gluing in compact subsets of initial data sets, and also discuss the case $\Lambda>0$ as well as matter models coupled to the Einstein equation; see also \cite{IsenbergMaxwellPollackGluing}. Carlotto--Schoen \cite{CarlottoSchoenData} gave another striking extension of this method, producing asymptotically flat initial data which are nontrivial ($\gamma$ Euclidean, $k=0$) only in arbitrary (noncompact!) \emph{cones} in $\R^3$.

We also mention Cortier's work \cite{CortierKdSGluing} on gluing exact Kerr--de~Sitter ends to solutions with asymptotically KdS ends, generalizing from the Kottler--Schwarzschild--de~Sitter case studied in \cite{ChruscielPollackKottler}. The latter results are very different from Theorem~\ref{ThmIBaby} as they concern the periodic ends of the level set $t=0$ in the maximal analytic extension of SdS and KdS spacetimes (called Delaunay ends in the SdS case). One can construct many-black-hole initial data sets, with a \emph{finite} number of black holes, from \cite{ChruscielPollackKottler,CortierKdSGluing} by identifying two isometric (in particular, sufficiently far apart) copies of the fundamental domain of the maximally extended SdS or KdS data set glued in near spatial infinity. In the case of exact SdS or KdS data sets, the resulting spacetime is a quotient of the maximal analytic extension by a suitable discrete translational symmetry; in particular, the future conformal boundary has several connected components, each of which is an interval times $\Sph^2$.

\subsection{Scattering problems on asymptotically simple spacetimes; gluing in de~Sitter space}
\label{SsIdS}

As discovered by Friedrich~\cite{FriedrichDeSitterPastSimple}, the `constraint equations' at the conformal boundary of an asymptotically simple spacetime\footnote{A manifold $M$ with boundary, and a metric $g$ on $M^\circ$ satisfying~\eqref{EqIEin} such that, for a boundary defining function $\tau$, the `unphysical metric' $\tau^2 g$ is a smooth Lorentzian metric on $M$, with $\pa M$ spacelike when $\Lambda>0$.} with $\Lambda>0$ simplify dramatically compared to~\eqref{EqIConstr}. Indeed, fixing a Riemannian 3-manifold $(S,h)$, the degrees of freedom are two scalar functions as well as a symmetric 2-tensor $k$ on $S$ satisfying the \emph{linear} equations
\begin{equation}
\label{EqIConstr2}
  \tr_h k=0,\quad \delta_h k=0.
\end{equation}
Given these data, one can construct an asymptotically simple solution $(M,g)$ of~\eqref{EqIEin} so that $S=\pa M$ and $h$ is the restriction to $\pa M$ of a suitable conformal multiple of $g$; the tensor $k$ is equal to certain components of the rescaled Weyl tensor of $g$ at $\pa M$. (We remark here that Dafermos--Holzegel--Rodnianski \cite{DafermosHolzegelRodnianskiKerrBw} gave a scattering construction of black holes settling down exponentially fast to a subextremal Kerr metric by solving a characteristic Cauchy problem `backwards' with cosmological constant $\Lambda=0$; see also \cite{RendallCharacteristic,LukCharacteristic}.)

The linear nature of the constraints~\eqref{EqIConstr2} suggest a simple way of gluing pieces of asymptotically simple spacetimes into de~Sitter space. Indeed, on Riemannian manifolds $(S,h)$ of dimension 3 and higher, the divergence operator on trace-free symmetric 2-tensors is underdetermined, and one can solve the divergence equation
\begin{equation}
\label{EqIDiv}
  \delta_h k'=f,\qquad \tr_h k'=0,
\end{equation}
in such a way that the support of $k'$ is contained in a small neighborhood of $\supp f$. This requires that there is no obstruction, i.e.\ $f$ must be orthogonal to the cokernel---the space of conformal Killing vector fields. Solvability then follows from a general result by Delay \cite{DelayCompact}. Thus, naively gluing many SdS black holes into neighborhoods of points $p_1,\ldots,p_N\in I^+=\Sph^3$ via a partition of unity, the constraints~\eqref{EqIConstr2} will typically be violated for the induced data $k$, and with $h=g_{\Sph^3}$; one can, however, correct $k$ by a 2-tensor $k'$ supported away from the points $p_i$ assuming the obstruction vanishes for $f=-\delta_h k$, which precisely leads to the balance condition~\eqref{EqIBabyBalance}; see~\S\ref{SsSO}.

We use the gluing problem as an opportunity to give a self-contained treatment of the scattering problem---the construction of a spacetime solution of~\eqref{EqIEin} from asymptotic data at $I^+$---in this specific context. Rather than using Friedrich's conformal Einstein field equations, see \cite[\S2]{FriedrichDeSitterPastSimple}, in which one solves for quantities derived from the metric tensor, we directly construct the metric as a Lorentzian \emph{0-metric} (uniformly degenerate metric), following the terminology of Mazzeo--Melrose \cite{MazzeoMelroseHyp}; see~\S\ref{SsSG}--\ref{SsSC}. As demonstrated by Vasy \cite[Theorem~5.5]{VasyWaveOndS}, solutions of linear wave equations on a spacetime with asymptotically de~Sitter type 0-metrics can be constructed from scattering data in Taylor series at $I^+$ using regular-singular point ODE methods; the remaining error, which vanishes to all orders at $I^+$, is solved away by solving a wave equation with such essentially trivial forcing. (Similar constructions are fairly standard in the Riemannian context on conformally compact or asymptotically hyperbolic metrics, see e.g.\ Fefferman--Graham \cite{FeffermanGrahamAmbient,FeffermanGrahamAmbientBook} and Graham--Zworski \cite{GrahamZworskiScattering}.)

In our gluing problem, this approach does not work directly. Indeed, calling the naively glued metric from the previous paragraph $g_0$, the leading order term of the resulting error $\Ric(g_0)-\Lambda g_0$ is of size $\cO(\tau^4)$ as a 2-tensor expressed in terms of $d\tau/\tau$ and sections of $T^*\Sph^3/\tau$ (and supported away from the points $p_i$) where $\tau=\cos s$ is a boundary defining function of $M$; the degenerate nature of (the linearization of) the Einstein vacuum equation prevents us from solving this error away using a metric correction of the same size. Instead, we need to use a metric correction of size $\cO(\tau^3)$ which does not produce any $\tau^3$ error terms (i.e.\ lies in the kernel of the indicial operator of the linearization of~\eqref{EqIEin}); in order for it to solve away the $\tau^4$ error, one needs to solve an equation of the form~\eqref{EqIDiv}.

To continue the construction, we use the now fairly precise glued metric, called $g^0$, as a background metric in a generalized harmonic gauge, similarly to~\cite{GrahamLeeConformalEinstein}, and solve the gauge-fixed Einstein equation (see Definition~\ref{DefSG}), first in Taylor series in~\S\ref{SsSG} (similarly now to the scalar wave equation case discussed in \cite{VasyWaveOndS}), and then nonlinearly by solving a quasilinear wave equation with rapidly decaying (at $\pa M$) forcing in~\S\ref{SsSC}. We show that the resulting metric solves the Einstein vacuum equation by using the usual argument based on the second Bianchi identity and the propagation of the gauge condition. In this final step, the sufficiently rapid vanishing of the gauge condition, a 1-form on spacetime, at $\pa M$ replaces the vanishing of the Cauchy data of this 1-form in the standard short-time theory.

Our gluing method is very flexible. For instance, one can glue any number of KdS black holes into the upper half space model $[0,1)_{\tilde\tau}\times\R^3_{\tilde x}$, $g_\dS=\tilde\tau^{-2}(-d\tilde\tau^2+d\tilde x^2)$, of de~Sitter space without any balance conditions if one allows for the solution to be sufficiently large at spatial infinity; in this case, we can of course only guarantee the existence of the nonlinear solution of~\eqref{EqIEin} in a neighborhood of $\tilde\tau=0$ which may shrink as $|\tilde x|\to\infty$. More generally, one can glue any metrics suitably asymptotic to de~Sitter space into de~Sitter space. See~\S\ref{SsSNc} for more on this. In particular, one may be able to glue several \emph{dynamical} KdS black holes together once their behavior is understood globally; see the work \cite{HintzVasyKdSStability} by the author with Vasy for the stability of the KdS exterior, and Schlue's ongoing project \cite{SchlueWeylDecay,SchlueOpticaldS} (building on his prior \cite{SchlueCosmological}) on the stability of the cosmological region.

We remark that, as another application of our approach, the polyhomogeneous formal solutions of~\eqref{EqIEin} constructed by Fefferman--Graham can be corrected to true (asymptotically de~Sitter like) solutions near the future conformal boundary; see Remark~\ref{RmkSCFormal}. This was previously shown by Rodnianski--Shlapentokh-Rothman \cite{RodnianskiShlapentokhRothmanSelfSimilar}.

\begin{rmk}
  We expect our methods to generalize in a straightforward manner to all higher dimensions, including to odd-dimensional spacetimes to which neither Friedrich's analysis nor the extensions by Anderson and Chru\'sciel apply \cite{AndersonStabilityEvenDS,AndersonChruscielSimple}.
\end{rmk}

\begin{rmk}
  It would be interesting to perform similar gluing constructions for Einstein--matter systems such as the Einstein--Maxwell equations, thus generalizing the family of Kastor--Traschen spacetimes. See also \cite{FriedrichEinsteinMaxwellYangMills}.
\end{rmk}

\subsection{Outline of the paper}

In~\S\ref{S0}, we recall relevant aspects of 0-analysis, i.e.\ the analysis of 0-metrics and associated uniformly degenerate differential operators. In~\S\ref{SS}, we present the details of the gluing construction for multi-SdS spacetimes; in~\S\ref{SK}, we extend this to the KdS case. Throughout the paper, the cosmological constant will be a fixed number
\[
  \Lambda > 0.
\]

\subsection*{Acknowledgments}

Part of this research was conducted during the period I served as a Clay Research Fellow. I would like to thank Maciej Zworski and Sara Kali\v{s}nik for their enthusiasm and support, and Richard Melrose and Andr\'as Vasy for discussions on a related project. I am also grateful to Piotr Chru\'sciel for helpful suggestions. This material is based upon work supported by the National Science Foundation under Grant No.\ DMS-1440140 while I was in residence at the Mathematical Sciences Research Institute in Berkeley, California, during the Fall 2019 semester.

\section{Analysis of uniformly degenerate metrics}
\label{S0}

We recall natural vector bundles associated with uniformly degenerate geometries in~\S\ref{Ss0B} and describe de~Sitter space from this point of view; associated differential operators are discussed in~\S\ref{Ss0D}. In~\S\ref{Ss0E}, we discuss the case of the Einstein vacuum equations in detail.

\subsection{Rescaled tangent and cotangent bundles; de~Sitter space}
\label{Ss0B}

Let $M$ be a smooth $(n+1)$-dimensional manifold with boundary $\pa M\neq\emptyset$; the space of smooth vector fields on $M$ is denoted $\cV(M)=\CI(M;T M)$. Following Mazzeo--Melrose \cite{MazzeoMelroseHyp}, we define the space
\[
  \cV_0(M) := \{ V\in\cV(M) \colon V(p)=0\ \forall\,p\in\pa M \}
\]
of \emph{0-vector fields} (or \emph{uniformly degenerate vector fields}); this is a Lie subalgebra of $\cV(M)$. If $\tau\in\CI(M)$ denotes a boundary defining function, i.e.\ $\pa M=\tau^{-1}(0)$ and $d\tau\neq 0$ on $\pa M$, then $\cV_0(M)=\tau\cV(M)$. In local coordinates $[0,\infty)_\tau\times\R^n_x$, the space $\cV_0(M)$ is the $\CI(M)$-span of the $n+1$ vector fields
\[
  \tau\pa_\tau,\ \tau\pa_{x^i},\ i=1,\ldots,n.
\]
Together, these provide a smooth frame of a vector bundle ${}^0 T M$, called \emph{0-tangent bundle}, which is \emph{nondegenerate} down to $\tau=0$. Thus, for $z\in M$, there is a natural map ${}^0 T_z M\to T_z M$ which is an isomorphism for $z\in M^\circ$. A section $V\in\CI(M;{}^0 T M)$ restricts to a smooth vector field on $M^\circ$ which extends smoothly to a vector field on $M$. This identifies $\cV_0(M)=\CI(M;{}^0 T M)$.

The dual bundle of ${}^0 T M$ is called the \emph{0-cotangent bundle} ${}^0 T^*M$. In local coordinates near $\pa M$, a local frame is given by
\[
  \frac{d\tau}{\tau},\ \frac{d x^i}{\tau},\ i=1,\ldots,n.
\]
These are smooth and nonvanishing down to $\tau=0$.

\begin{definition}
\label{Def0BMetric}
  A \emph{Lorentzian 0-metric} (or \emph{uniformly degenerate Lorentzian metric}) $g$ on $M$ of class $\CI$ is a smooth section $g\in\CI(M;S^2\,{}^0 T^*M)$ which has signature $(n,1)$ at every point of $M$.
\end{definition}

In local coordinates, a smooth Lorentzian 0-metric can be written as
\[
  g = \tau^{-2}\Biggl(g_{0 0}d\tau^2 + 2\sum_{i=1}^n g_{0 i}d\tau\otimes_s d x^i + \sum_{i,j=1}^n g_{i j}d x^i\otimes_s d x^j \Biggr),
\]
with the $g_{\mu\nu}$ smooth functions of $(\tau,x)$, and $(g_{\mu\nu})_{\mu,\nu=0,\ldots,n}$ having signature $(n,1)$. Note that $\tau^2 g\in\CI(M;S^2 T^*M)$ is a smooth Lorentzian metric on $M$ in the usual sense. In particular, the class of metrics $g$ for which $\pa M$ is spacelike for the metric $\tau^2 g$ is well-defined, and independent of the choice of boundary defining function $\tau$; we shall only be concerned with such metrics in the present paper. The Riemannian metric induced on $\pa M$ by $\tau^2 g$ does depend on $\tau$, but its conformal class is well-defined.

The prime example for us is the de~Sitter spacetime in $3+1$ dimensions, with cosmological constant $\Lambda>0$. It can be defined as the cylinder\footnote{Just this one time, we also include the past conformal boundary.}
\begin{subequations}
\begin{equation}
\label{Eq0BdSGlobal}
  M=[-\pi/2,\pi/2]_s\times\Sph^3,\quad
  g_\dS = \frac{3}{\Lambda}\cdot\frac{-d s^2+g_{\Sph^3}}{\cos^2 s},
\end{equation}
whose interior is conformally diffeomorphic to a slab inside the Einstein universe $(\R_s\times\Sph^3,-d s^2+g_{\Sph^3})$; here $g_{\Sph^3}$ is the standard metric on $\Sph^3$. The metric $g_\dS$ is a solution of the Einstein vacuum equation~\eqref{EqIEin}. To see that $g_\dS$ has the required form near $s=\pi/2$, let us take $\tau=\cos s$ near $s=\pi/2$; then
\begin{equation}
\label{Eq0BdSGlobal2}
  g_\dS = (3/\Lambda)\tau^{-2}\bigl( -(1-\tau^2)^{-1} d\tau^2 + g_{\Sph^3} \bigr)\quad\text{on}\ \ [0,\infty)_\tau\times\Sph^3_\psi.
\end{equation}
\end{subequations}
Note that $\tau^2 g_\dS|_{\pa M}=(3/\Lambda)g_{\Sph^3}$ is a Riemannian metric, thus $\pa M=\Sph^3\sqcup\Sph^3$ is spacelike with respect to $g_\dS$.

Other forms of the de~Sitter metric are useful for calculations. Regarding $\Sph^3$ as the unit sphere $\Sph^3\subset\R^4=\R\times\R^3$, we define the map
\begin{subequations}
\begin{equation}
\label{Eq0BdSNoncpt}
\begin{aligned}
  &[0,1) \times \R^3 \ni (\tilde\tau, \tilde x) \mapsto (\tau,\psi) \in [0,\infty) \times \Sph^3, \\
  &\qquad \tau = \Bigl(\Bigl(\frac{1-(\tilde\tau^2-|\tilde x|^2)}{2\tilde\tau}\Bigr)^2+1\Bigr)^{-1/2},\quad
    \psi = \frac{\tau}{\tilde\tau}\Bigl(\frac{1+\tilde\tau^2-|\tilde x|^2}{2},\tilde x\Bigr) \in \Sph^3
\end{aligned}
\end{equation}
from part of the upper half space into de~Sitter space~\eqref{Eq0BdSGlobal2}; here $|\cdot|$ is the Euclidean norm. The de~Sitter metric then takes the form
\begin{equation}
\label{Eq0BdSNoncpt2}
  g_\dS = \frac{3}{\Lambda}\cdot\frac{-d\tilde\tau^2+d\tilde x^2}{\tilde\tau^2}\quad\text{on}\ \ M_u := [0,\infty)_{\tilde\tau} \times \R^3_{\tilde x}.
\end{equation}
\end{subequations}
See \cite[\S6.1]{HintzZworskiHypObs} for these and related calculations (in particular, relating both~\eqref{Eq0BdSGlobal2} and \eqref{Eq0BdSNoncpt2} to the one-sheeted hyperboloid in $(1+(n+1))$-dimensional Minkowski space which is isometric to global de~Sitter space); they imply that the map~\eqref{Eq0BdSNoncpt} composed with $(\tau,\psi)\mapsto(s,\psi)$, $s=\arccos\tau$ in the coordinates~\eqref{Eq0BdSGlobal}, extends analytically to a map $[0,\infty)_{\tilde\tau}\times\R^3_{\tilde x}\to M$ whose image is the complement of the backward causal cone from the point $(-1,0)\in\Sph^3$ at $s=\pi/2$; see \cite[Figure~7]{HintzZworskiHypObs}.

Finally, introducing polar coordinates $\tilde x=\tilde R\tilde\omega$, $\tilde R=|\tilde x|\geq 0$, $\tilde\omega\in\Sph^2$, and putting
\begin{subequations}
\begin{equation}
\label{Eq0BdSStatic}
  (t,r,\omega) = \bigl(-\half\sqrt{\Lambda/3}\log(\tilde R^2-\tilde\tau^2),\,\sqrt{\Lambda/3}\tilde\tau^{-1}\tilde R,\,\tilde\omega)
\end{equation}
in the \emph{cosmological region} $\tilde R>\tilde\tau$, we have
\begin{equation}
\label{Eq0BdSStatic2}
  g_\dS = -\Bigl(\frac{\Lambda r^2}{3}-1\Bigr)^{-1}d r^2 + \Bigl(\frac{\Lambda r^2}{3}-1\Bigr)d t^2 + r^2 g_{\Sph^2}.
\end{equation}
This is a smooth 0-metric on a compactification of $(\sqrt{3/\Lambda},\infty)_r\times\R_t\times\Sph^2_\omega$; indeed, letting $\tau_s=r^{-1}$, and defining
\begin{equation}
\label{Eq0BdSStatic3Mfd}
  M_s := [0,\sqrt{\Lambda/3})_{\tau_s} \times \R_t \times \Sph^2_\omega,
\end{equation}
we have
\begin{equation}
\label{Eq0BdSStatic3}
  g_\dS = \tau_s^{-2}\Bigl(-\bigl(\Lambda/3-\tau_s^2\bigr)^{-1}d\tau_s^2 + \bigl(\Lambda/3-\tau_s^2\bigr)d t^2 + g_{\Sph^2}\Bigr) \in \CI(M_s;S^2\,{}^0 T^*M_s).
\end{equation}
\end{subequations}
The metric induced on $\tau_s=0$ (factoring out the overall scalar factor $\Lambda/3$) is
\begin{equation}
\label{Eq0BdSStaticBdy}
  h_s:=(\Lambda/3)\tau_s^2 g_\dS|_{\pa M}=\frac{\Lambda^2}{9}d t^2+\frac{\Lambda}{3}g_{\Sph^2}.
\end{equation}

We remark that $\tau$, $\tilde\tau$, and $\tau_s$ are equivalent defining functions on the overlaps of the various coordinate charts.

\subsection{Differential operators, function spaces}
\label{Ss0D}

Geometric operators associated with a 0-metric $g$ on an $(n+1)$-dimensional manifold $M$ are examples of \emph{0-differential operators}. Concretely, using abstract index notation, we shall in particular deal with the wave operator on bundles,
\[
  \Box_g u = -u_{;\kappa}{}^\kappa,\quad
  (\Box_g\dot g)_{\mu\nu} = -\dot g_{\mu\nu;\kappa}{}^\kappa,
\]
the divergence and (trace-free) symmetric gradient
\[
  \delta_g\omega = -\omega_{\kappa;}{}^\kappa,\quad
  (\delta_g\dot g)_\mu = -\dot g_{\mu\kappa;}{}^\kappa,\quad
  (\delta_g^*\omega)_{\mu\nu} = \half(\omega_{\mu;\nu}+\omega_{\nu;\mu}),\quad
  \delta_{g,0}^*\omega := \delta_g^*\omega + \tfrac{1}{n+1}g\delta_g\omega,
\]
as well as the `trace reversal\footnote{One has $\tr_g\circ\sfG_g=-\tr_g$ only for $n+1=4$.} operator' on 2-tensors,
\[
  \sfG_g\dot g := \dot g - \half g(\tr_g \dot g).
\]
We define the space $\Diff_0^m(M)$ of $m$-th order 0-differential operators to consist of all locally finite linear combinations of up to $m$-fold products of 0-vector fields. Then $\Box_g\in\Diff_0^2(M)$ for the scalar wave operator, $\Box_g\in\Diff_0^2(M;S^2\,{}^0 T^*M)$ for the tensor wave operator acting on symmetric 2-tensors, $\delta_g^*\in\Diff_0^1(M;{}^0 T^*M,S^2\,{}^0 T^*M)$, and so forth. For instance, for the metric~\eqref{Eq0BdSNoncpt2} in $3+1$ dimensions, we have
\[
  3\Lambda^{-1}\Box_g = (\tilde\tau\pa_{\tilde\tau})^2 - 3\tilde\tau\pa_{\tilde\tau} + \tilde\tau^2\Delta_{\tilde x},\quad \Delta_{\tilde x}:=-\sum_{i=1}^n \pa_{\tilde x^i}^2;
\]
for the other operators, we will give explicit expressions in~\S\ref{Ss0E}.

Associated with any 0-differential operator $A\in\Diff_0^m(M)$ is its \emph{indicial family} (see also \cite[\S2]{MazzeoEdge}) $I(A,\lambda)\in\CI(\pa M)$, $\lambda\in\C$, which is defined by
\[
  A(\tau^\lambda u) = \tau^\lambda I(A,\lambda)u + \cO(\tau^{\lambda+1}),\quad u\in\CI(M),
\]
for any defining function $\tau$; this is independent of the choice of defining function. Concretely,
\begin{equation}
\label{Eq0BIndFam}
  A=\sum_{i+|\alpha|\leq m}a_{i\alpha}(\tau,x)(\tau\pa_\tau)^i(\tau\pa_x)^\alpha\ \Longrightarrow \ I(A,\lambda)=\sum_{i\leq m}a_{i 0}(0,x)\lambda^i.
\end{equation}
Thus, $I(A,\lambda)$ is a polynomial of degree $m$ in $\lambda$, depending smoothly on $x\in\pa M$. We call the roots of the polynomial $\lambda\mapsto I(A,\lambda)(x)$ the \emph{indicial roots} of $A$; if they are independent of $x$, we say that \emph{$A$ has constant indicial roots}.

If $A\in\Diff_0^k(M;E,F)$ acts between sections of vector bundles $E,F\to M$, we define $I(A,\lambda)\in\CI(\pa M;\Hom(E,F)|_{\pa M})$ similarly; the indicial roots of $A$ are then those $\lambda$ (depending on $x\in\pa M$) for which $I(A,\lambda)$ fails to be invertible.

Lower order terms of $A$ as in~\eqref{Eq0BIndFam} can be defined upon fixing a collar neighborhood $[0,\eps)_\tau\times\pa M$ of $\pa M$: writing
\[
  A\equiv\sum_{k=0}^m \tau^k \sum_{\genfrac{}{}{0pt}{}{i+|\alpha|\leq m}{|\alpha|=k}} a^{(k)}_{i\alpha}(x)(\tau\pa_\tau)^i\pa_x^\alpha.
\]
modulo terms of the form $\tau^{k+1}a(\tau,x)(\tau\pa_\tau)^i\pa_x^\alpha$, $a\in\CI$,\footnote{The reader familiar with b-analysis \cite{MelroseAPS} will recognize this as the Taylor expansion of $A$ into dilation-invariant (with respect to $\tau$) b-differential operators on $[0,\eps)_\tau\times\pa M$.} we define
\begin{equation}
\label{Eq0BIndFamlot}
  I(A[\tau^k],\lambda) := \sum_{\genfrac{}{}{0pt}{}{i+|\alpha|\leq m}{|\alpha|=k}} a_{i\alpha}(x)\lambda^i\pa_x^\alpha \in \Diff^k(\pa M).
\end{equation}
Thus, $I(A[1],\lambda)=I(A,\lambda)$. If $A$ acts between sections of vector bundles $E,F\to M$, the same works, with $a_{i\alpha}^{(k)}\in\CI(\pa M;\Hom(E,F)|_{\pa M})$, upon choosing an identification of $E,F$ in the collar neighborhood with pullbacks of $E|_{\pa M},F|_{\pa M}$ along the projection $[0,\eps)\times\pa M\to\pa M$.

We record some standard calculations involving the indicial family. If $A\in\Diff_0^m(M;E,F)$ and $\lambda\in\C$ are such that $\ker I(A,\lambda)$ is a $\CI$ vector subbundle of $E|_{\pa M}\to\pa M$, then $A(\tau^\lambda u)=\cO(\tau^{\lambda+1})$ for all $u\in\CI(M)$ with $u|_{\pa M}\in\CI(\pa M;\ker I(A,\lambda))$. For such $u$, we moreover have
\begin{equation}
\label{Eq0BIndLog}
  A(\tau^\lambda(\log\tau)u) = \tau^\lambda \pa_\lambda I(A,\lambda)u + \cO(\tau^{\lambda+1}\log\tau);
\end{equation}
this can be seen by differentiating the relationship $A(\tau^\lambda v)=\tau^\lambda I(A,\lambda)v+\tau^{\lambda+1}\tilde v$ (with $\tilde v\in\CI(M)$ depending smoothly on $\lambda$) in $\lambda$ and plugging in $v=u$. We also record that in a collar neighborhood of $\pa M$, we have, for such $u$, $A(\tau^\lambda u)=\tau^{\lambda+1}I(A[\tau],\lambda)u+\cO(\tau^{\lambda+2})$.

The $L^2$-based function spaces corresponding to 0-analysis are \emph{weighted 0-Sobolev spaces}
\[
  \tau^m H_{0,\loc}^k(M) = \{ \tau^m u \colon u\in H_{0,\loc}^k(M) \}.
\]
For $k=0$, we define $H_{0,\loc}^0(M)=L^2_\loc(M)$ to be the space of locally\footnote{On $M$, thus this does encode uniformity down to compact subsets of $\pa M$.} square integrable functions on $M$ relative for a smooth 0-density, i.e.\ a smooth positive section of the 0-density bundle $|\Lambda^{n+1}\,{}^0 T^*M|$; in local coordinates, such a density takes the form $a(\tau,x)|\frac{d\tau\,d x}{\tau^{n+1}}|$ with $0<a\in\CI$, a typical example being the volume density $|d g|$ of a Lorentzian 0-metric $g$. For $k\in\N$, we define $H_{0,\loc}^k(M)$ to consist of all $u\in L^2(M)$ so that $P u\in L^2(M)$ for all $P\in\Diff_0^k(M)$. If $M$ is compact, the space $\tau^m H_0^k(M)=\tau^m H_{0,\loc}^k(M)$ carries the structure of a Hilbert space. More generally, if $M$ is noncompact and $\Omega\subset M$ is open with compact closure, then
\begin{equation}
\label{Eq0BSobolevExt}
  \tau^m\Hext_0^k(\Omega):=\{ u|_\Omega \colon u\in\tau^m H_0^k(M) \}
\end{equation}
is a Hilbert space.

For compact $M$, we can characterize the space $H_0^m(M)$ using a covering of $M$ by `uniformly degenerating cubes' as follows: if a distribution $u$ is supported in a coordinate patch $[0,2)_\tau\times\R^n_x$, and in fact in $\tau\leq 1$, $|x|\leq 1$, then\footnote{We write $A\sim B$ to mean the existence of a constant $C>1$, independent of $u$, so that $C^{-1}B\leq A\leq C B$.}
\begin{align*}
  &\|u\|_{H_0^m(M)}^2 \sim \sum_{k=0}^\infty \sum_{\genfrac{}{}{0pt}{}{\alpha\in\Z^n}{|\alpha|\leq 2^k}} \|u_{k,\alpha}\|_{H^m([-1/2,1/2]^{n+1})}^2, \\
  &\qquad u_{k,\alpha}(T,X) := u\bigl(2^{-k}(1+T), 2^{-k}(\alpha+X)\bigr),\quad (T,X)\in\R\times\R^n,\ |T|,|X|\leq\half.
\end{align*}
Note that $u_{k,\alpha}$ sees $u$ on a cube of size $2^{-k}$ centered at a point at a distance $2^{-k}$ from the boundary, and $\pa_T,\pa_X$ are of the same size as $\tau\pa_\tau$, $\tau\pa_x$. We leave the notational changes required to drop the support condition to the reader; see also \cite[Proof of Corollary~(3.23)]{MazzeoEdge}. An important consequence of this characterization is that algebra properties of Sobolev spaces on $\R^n$ immediately carry over to 0-Sobolev spaces; in particular:

\begin{lemma}
\label{Lemma0BSobAlg}
  On an $(n+1)$-dimensional compact manifold $M$ with boundary, and for $k>(n+1)/2$, the space $H_0^k(M)$ is an algebra. More generally, we have
  \[
    u_j\in\tau^{m_j}H_0^{k_j}(M),\ j=1,2\ \Longrightarrow\ u_1 u_2\in\tau^{m_1+m_2}H_0^{k_1+k_2}(M).
  \]
\end{lemma}

Solutions of uniformly degenerate equations often have better regularity and are \emph{conormal}, for instance as shown for solutions of the wave equation on de~Sitter type spaces in \cite{VasyWaveOndS}. For $\alpha\in\R$, we define the space of \emph{conormal functions} relative to $\tau^\alpha L^\infty(M)$ by
\[
  \cA^\alpha(M) := \{ u\in \CI(M^\circ) \colon P(\tau^{-\alpha}u) \in L^\infty(M)\ \forall\,P\in\Diffb(M) \},
\]
where $\tau\in\CI(M)$ is a boundary defining function, and $\Diffb(M)$ is the space of all \emph{b-differential operators} on $M$: locally, these are finite products of the vector fields $\tau\pa_\tau$ and $\pa_{x^j}$ with $\CI(M)$ coefficients. The space
\begin{equation}
\label{Eq0BConormalDiff}
  \cA^\alpha\Diff_0^m(M)
\end{equation}
of 0-differential operators with conormal coefficients consists of all locally finite linear combinations of differential operators of the form $a P$, $a\in\cA^\alpha(M)$, $P\in\Diff_0^m(M)$.

\subsection{Einstein vacuum equation and its linearization}
\label{Ss0E}

We make some general observations about the following nonlinear operator for 0-metrics:

\begin{definition}
\label{Def0EEinOp}
  For a Lorentzian metric $g$ on a manifold $M$, define
  \begin{equation}
  \label{Eq0EEinOp}
    P_0(g) := 2(\Ric(g)-\Lambda g).
  \end{equation}
\end{definition}

For definiteness, we now work in $3+1$ dimensions on the spacetime manifold
\[
  M = [0,1)_\tau\times X,\qquad \dim X=3,
\]
where $X$ is a $3$-dimensional manifold without boundary such as $\R^3$ or $\Sph^3$; the boundary $\pa M=\tau^{-1}(0)$ will play the role of the future conformal boundary. This product structure allows us to identify differential operators (in particular: vector fields) on $X$ with `spatial' differential operators on $M$. In particular, we can pull back $T X$ (along the projection $M\to X$) to a bundle over $M$, still denoted $T X$, which allows us to split ${}^0 T M$ as
\[
  {}^0 T M = \R e_0 \oplus \tau T X,\quad e_0=\tau\pa_\tau;
\]
that is, we identify a 0-vector field $u e_0 + \tau V$, $u\in\CI(M)$, $V\in\CI(M;T X)\subset\cV(M)$, with the pair $(u,V)$. This induces splittings
\begin{equation}
\label{Eq0ESplit}
\begin{split}
  {}^0 T^*M &= \R e^0 \oplus \tau^{-1}T^*X,\quad e^0=\frac{d\tau}{\tau}, \\
  S^2\,{}^0 T^*M &= \R(e^0)^2 \oplus \bigl(2 e^0\otimes_s\tau^{-1}T^*X\bigr) \oplus \tau^{-2}S^2 T^*X;
\end{split}
\end{equation}
that is, we identify a section $u(e^0)^2+2 e^0\otimes\tau^{-1}\omega+\tau^{-2}k$ of $S^2\,{}^0 T^*M$ with the triple $(u,\omega,k)$, where $u\in\CI(M)$, $\omega\in\CI(M;T^*X)$, $k\in\CI(M;S^2 T^*X)$. Given a Riemannian metric $h$ on $X$, we can split $S^2 T^*X$ into pure trace ($\R h$) and trace-free parts ($\ker\tr_h$), thereby refining~\eqref{Eq0ESplit} to
\begin{equation}
\label{Eq0ESplitTr}
  S^2\,{}^0 T^*M = \R(e^0)^2 \oplus \bigl(2 e^0\otimes_s\tau^{-1}T^*X\bigr) \oplus \R\tau^{-2}h \oplus \tau^{-2}\ker\tr_h.
\end{equation}
We shall denote the components of $\dot g\in S^2\,{}^0 T^*M$ in the four summands in~\eqref{Eq0ESplitTr} by $\dot g_{N N}\in\R$ (normal-normal), $\dot g_{N T}\in T^*X$ (normal-tangential), $\dot g_{T T 1}\in\R$ (tangential-tangential, pure trace), $\dot g_{T T 0}\in\ker\tr_h$ (tangential-tangential, trace-free).

We shall first study geometric operators associated with a product metric
\begin{equation}
\label{Eq0EMetricProd}
  g = (3/\Lambda)\frac{-d\tau^2+h(x,d x)}{\tau^2}.
\end{equation}
We denote the exterior derivative on $X$, pulled back to a spatial operator on $M$, by $d_X$.

\begin{lemma}
\label{Lemma0EOps}
  In the splittings~\eqref{Eq0ESplit}, we have
  \[
    \delta_g^* = \begin{pmatrix} e_0 & 0 \\ \half\tau d_X & \half(1+e_0) \\ h & \tau\delta_h^* \end{pmatrix}, \quad
    3\Lambda^{-1}\delta_g = \begin{pmatrix} e_0-3 & \tau\delta_h & -\tr_h \\ 0 & e_0-4 & \tau\delta_h \end{pmatrix},
  \]
  and, as operators on symmetric 2-tensors,
  \[
    \sfG_g = \begin{pmatrix} \half & 0 & \half\tr_h \\ 0 & 1 & 0 \\ \half h & 0 & \sfG_h \end{pmatrix}, \quad
    3\Lambda^{-1}\Box_g = e_0^2 - 3 e_0 + \tau^2\Delta_h + \begin{pmatrix} -6 & 4\tau\delta_h & -2\tr_h \\ -2\tau d_X & -6 & 2\tau\delta_h \\ -2 h & -4\tau\delta_h^* & -2 \end{pmatrix}.
  \]
  Finally, if $R_{\kappa\lambda\mu\nu}$ and $\Ric_{\mu\nu}$ denote the Riemann curvature tensor and Ricci tensor of $g$, then the operator $\sR_g(u)_{\kappa\mu}=R^\nu{}_{\kappa\mu}{}^\rho u_{\nu\rho}+\half(\Ric_\kappa{}^\nu u_{\nu\mu}+\Ric_\mu{}^\nu u_{\kappa\nu})$ is equal to
  \[
    3\Lambda^{-1}\sR_g = \begin{pmatrix} 3 & 0 & \tr_h \\ 0 & 4 & 0 \\ h & 0 & 4-h\tr_h \end{pmatrix} + \tau^2\begin{pmatrix} 0 & 0 & 0 \\ 0 & \half\Ric(h) & 0 \\ 0 & 0 & \sR_h \end{pmatrix}.
  \]
\end{lemma}
\begin{proof}
  In local coordinates $x^1,x^2,x^3$ on $X$, and setting $e_i=\tau\pa_{x^i}$, $e^i=\tau^{-1}d x^i$, we compute $\nabla_{e_0}e^\mu=0$ for all $\mu$, $\nabla_{e_i}e^0=h_{i k}e^k$, and $\nabla_{e_i}e^k=\delta_i^k e_0-\tau\Gamma(h)_{i j}^k e^j$, where $\delta_i^k$ is the Kronecker delta. Moreover, we have $R^\lambda{}_{\kappa\mu}{}^\nu=0$ except for $3\Lambda^{-1} R^0{}_{k i}{}^0=h_{i k}$, $3\Lambda^{-1} R^0{}_{i 0}{}^j=\delta_i^j$, $3\Lambda^{-1} R^l{}_{0 i}{}^0=\delta_i^l$, $3\Lambda^{-1} R^l{}_{0 0}{}^j=h^{l j}$, $3\Lambda^{-1} R^l{}_{k i}{}^j=\delta_i^l\delta_k^j-h_{i k}h^{j l}+\tau^2 R(h)^l{}_{k i}{}^j$, where $h^{j l}=h^{-1}(d x^j,d x^k)$ denotes the coefficients of the dual metric of $h$; this gives $\Ric_{0 0}=-\Lambda$, $\Ric_{i 0}=0$, and $\Ric_{i j}=\Lambda h_{i j}+(\Lambda/3)\tau^2\Ric(h)_{i j}$. The expressions in the lemma follow from this by direct computation.
\end{proof}

The calculations in the proof imply that $\Ric(g)-\Lambda g\in\tau^2\CI(M;S^2\,{}^0 T^*M)$, i.e.\ any $g$ of the form~\eqref{Eq0EMetricProd} satisfies the Einstein equation modulo $\cO(\tau^2)$ errors.

In \cite{DeTurckPrescribedRicci,GrahamLeeConformalEinstein}, the linearization of $P_0$ is computed as
\begin{equation}
\label{Eq0ELinRic}
  L_{0,g}\dot g := D_g P_0(\dot g) = \Box_g - 2\delta_g^*\delta_g\sfG_g + 2\sR_g - 2\Lambda.
\end{equation}
Using Lemma~\ref{Lemma0EOps}, we compute the leading and subleading order behavior of $L_{0,g}$:

\begin{cor}
\label{Cor0EInd}
  For $g$ as in~\eqref{Eq0EEinOp}, and in the bundle splitting~\eqref{Eq0ESplitTr}, we have
  \begin{align*}
    3\Lambda^{-1}I(L_{0,g},\lambda) &= \begin{pmatrix} 3\lambda-6 & 0 & -\lambda(3\lambda-6) & 0 \\ 0 & 0 & 0 & 0 \\ 6-\lambda & 0 & -\lambda(6-\lambda) & 0 \\ 0 & 0 & 0 & \lambda(\lambda-3) \end{pmatrix}, \\
    3\Lambda^{-1}I(L_{0,g}[\tau],\lambda) &= \begin{pmatrix} 0 & 2(1-\lambda)\delta_h & 0 & 0 \\ 2 d_X & 0 & -2\lambda d_X & -\lambda\delta_h \\ 0 & \tfrac{2}{3}(\lambda-5)\delta_h & 0 & 0 \\ 0 & (4-2\lambda)\delta_{h,0}^* & 0 & 0 \end{pmatrix}.
  \end{align*}
\end{cor}
\begin{proof}
  In the calculation of $\delta_g^*\delta_g\sfG_g$, one needs to use $e_0\tau=\tau(e_0+1)$ as well as $\delta_h h=-d_X$ and $\delta_h^*=\delta_{h,0}^*-\frac{1}{3}h\delta_h$ to obtain the stated expression for $I(L_{0,g}[\tau],\lambda)$.
\end{proof}

Typically, metrics do have $\tau$-dependence. The following two lemmas describe the (linearized) Einstein operator for lower order perturbations of~\eqref{Eq0EMetricProd}.

\begin{lemma}
\label{Lemma0EEinLinPert}
  If $\alpha>0$, $\tilde g\in\cA^\alpha(M;S^2\,{}^0 T^*M)$, then $L_{0,g+\tilde g}-L_{0,g}\in\cA^\alpha\Diff_0^2(M;S^2\,{}^0 T^*M)$, see~\eqref{Eq0BConormalDiff}. If $\tilde g\in\tau^m\CI$ for some $m\in\N$, then $L_{0,g+\tilde g}-L_{0,g}\in\tau^m\Diff_0^2$.
\end{lemma}
\begin{proof}
   Using that the space $\cA^\alpha$ is a $\CI(M)$-module which is closed under multiplication, we have $(g+\tilde g)^{-1}-g^{-1}\in\cA^\alpha(M;S^2\,{}^0 T M)$. Hence $\sfG_{g+\tilde g}-\sfG_g\in\cA^\alpha(M;\End(S^2\,{}^0 T^*M))$. Similarly, indicating the metric in the notation for the Levi-Civita connection by a superscript,
  \[
    \nabla^{g+\tilde g}-\nabla^g\in\cA^\alpha\Diff_0^1(M;{}^0 T^*M;{}^0 T^*M\otimes{}^0 T^*M).
  \]
  This implies that $\delta_{g+\tilde g}^*-\delta_g^*\in\cA^\alpha\Diff_0^1(M;{}^0T^*M;S^2\,{}^0 T^*M)$, similarly for the other operators appearing in~\eqref{Eq0ELinRic}. For the proof of the second part of the lemma, replace $\cA^\alpha$ by $\tau^m\CI$.
\end{proof}

In particular, for $m\geq 2$, the indicial families $I(L_{0,g+\tilde g}[\tau^j],\lambda)$, $j=0,1$, are independent of $\tilde g$; likewise (suitably interpreted) for $\tilde g\in\cA^\alpha$, $\alpha>1$.

\begin{lemma}
\label{Lemma0EEin}
  With $P_0$ defined in~\eqref{Eq0EEinOp}, suppose $g\in\CI(M;S^2\,{}^0 T^*M)$. If $\alpha>0$, $\tilde g\in\cA^\alpha(M;S^2\,{}^0 T^*M)$, then $P_0(g+\tilde g)-P_0(g)-L_{0,g}\tilde g\in\cA^{2\alpha}(M;S^2\,{}^0 T^*M)$. Similarly, if $\tilde g\in\tau^m\C$, $m\in\N$, then $P_0(g+\tilde g)-P_0(g)-L_{0,g}\tilde g\in\tau^{2 m}\CI$.
\end{lemma}
\begin{proof}
  This follows similarly to the proof of Lemma~\ref{Lemma0EEinLinPert}. Since $P_0(g)$ and $L_{0,g}\tilde g$ capture all terms of $P_0(g+\tilde g)$ which are at most linear in $\tilde g$, the difference $P_0(g+\tilde g)-P_0(g)-L_{0,g}\tilde g$ only contains terms which are at least quadratic in $\tilde g$, hence its coefficients, as a 0-differential operator, have the stated decay.
\end{proof}

\section{Multi-Schwarzschild--de~Sitter spacetimes}
\label{SS}

In this section, we show how to glue several Schwarzschild--de~Sitter metrics into global de~Sitter space; we shall work near the future conformal boundary, hence on
\begin{equation}
\label{EqSMfd}
  M = [0,1)_\tau \times \Sph^3,\quad \tau=\cos s.
\end{equation}
The de~Sitter metric is of the form discussed in Lemmas~\ref{Lemma0EEinLinPert}--\ref{Lemma0EEin}. Indeed, we have
\begin{equation}
\label{EqSBdyMet}
  g_\dS\in 3\Lambda^{-1}\tau^{-2}(-d\tau^2+h)+\tau^2\CI,\qquad
  h=3\Lambda^{-1}g_{\Sph^3},
\end{equation}

We recall the Schwarzschild--de~Sitter (SdS) metric with mass $\bhm\in\R$:
\begin{equation}
\label{EqSMetric}
  g_\bhm = -\Bigl(\frac{\Lambda r^2}{3}-1+\frac{2\bhm}{r}\Bigr)^{-1}d r^2 + \Bigl(\frac{\Lambda r^2}{3}-1+\frac{2\bhm}{r}\Bigr)d t^2 + r^2 g_{\Sph^2}.
\end{equation}
We consider the metric~\eqref{EqSMetric} for $r>r_+$, where $r_+$ is the largest positive real root of $\Lambda r^2/3-1+2\bhm/r$ if one exists; otherwise, fix an arbitrary $r_+>0$. As in~\eqref{Eq0BdSStatic3}, we put $\tau_s=r^{-1}$, and thus $g_\bhm$ is a smooth 0-metric on
\begin{equation}
\label{EqSCoords}
  M_{\bhm,s} := [0,r_+^{-1})_{\tau_s} \times \R_t \times \Sph^2_\omega.
\end{equation}
Comparison with the de~Sitter metric, expressed in the same coordinates and on the manifold $M_s$ (see~\eqref{Eq0BdSStatic3Mfd}) by~\eqref{Eq0BdSStatic2} and \eqref{Eq0BdSStatic3} (thus $g_\dS=g_\bhm|_{\bhm=0}$), shows that
\begin{equation}
\label{EqSODiff}
  g_\bhm - g_\dS \in \tau_s^3\CI(M_s\cap M_{\bhm,s}; S^2\,{}^0 T^*M_s)
\end{equation}
in their common domain of definition. Note that at $\tau_s=0$, we have, in the upper half space coordinates~\eqref{Eq0BdSStatic}, $\tilde\tau=0$ and $\tilde R=e^{-t\sqrt{3/\Lambda}}$. In particular, $t\to\infty$ corresponds to $\tilde R\to 0$; let $p_0=(1,0,0,0)\in\Sph^3\subset\R^4$ denote the point $\tilde R=0$ inside $\tilde\tau=0$. Moreover, $t\to-\infty$ corresponds to $\tilde R\to\infty$, which on global de~Sitter space corresponds to the antipodal point $-p_0\in\Sph^3$ inside $\pa M$ by inspection of~\eqref{Eq0BdSNoncpt}.

In summary, by relating the coordinates in~\eqref{EqSCoords} to the semi-global de~Sitter manifold~\eqref{EqSMfd}, $g_\bhm$ can be regarded as gluing an SdS black hole into de~Sitter space at the point $p_0$ at the future conformal boundary $\tau=0$. Given a point $p\in\Sph^3$, choose a rotation $R\in SO(4)$ with $R p=p_0$; this induces a map $(\tau,\psi)\mapsto(\tau,R(\psi))$ on $M$. Pulling back $g_\bhm$ along this map, we obtain the metric
\begin{equation}
\label{EqSParamMetric}
  g_{p,\bhm},\quad p\in\Sph^3,\ \bhm\in\R,
\end{equation}
with $g_{p,\bhm}$ defined in a neighborhood of $U_p=\Sph^3\setminus\{p,-p\}$. See Figure~\ref{FigISdS}.

\begin{definition}
\label{DefSBalance}
  Let $N\in\N$. We say that $\{(p_i,\bhm_i)\colon i=1,\ldots,N\}\subset\Sph^3\times\R$ is \emph{balanced} if the $p_i$ are pairwise distinct and if, regarding $\Sph^3$ as the unit sphere inside $\R^4$, the following relation holds:
  \[
    \sum_{i=1}^N \bhm_i p_i = 0 \in \R^4.
  \]
\end{definition}

We can now state our main theorem:

\begin{thm}
\label{ThmSGlue}
  Let $N\in\N$, and suppose $\{(p_1,\bhm_1),\ldots,(p_N,\bhm_N)\}\subset\Sph^3\times\R$ is balanced. Suppose $V_{p_i}\subset U_{p_i}$ is a ball around $p_i$ with the point $p_i$ removed, and suppose $\ol{V_{p_i}}\cap\ol{V_{p_j}}=\emptyset$ for $i\neq j$. Then there exist a neighborhood $U$ of $\pa M\setminus\{p_1,\ldots,p_N\}$ and a Lorentzian 0-metric $g\in\CI(U;S^2\,{}^0 T_U^*M)$ with the following properties:
  \begin{enumerate}
  \item\label{ItSGlueEin} $g$ satisfies the Einstein vacuum equation $\Ric(g)-\Lambda g=0$;
  \item\label{ItSGlueSdS} near $V_{p_i}$, we have $g=g_{p_i,\bhm_i}$;
  \item\label{ItSGluePert} $g$ is $\cO(\tau^3)$-close to the de~Sitter metric: $g-g_\dS\in\tau^3\CI(U;S^2\,{}^0 T_U^*M)$.
  \end{enumerate}
\end{thm}

See Figure~\ref{FigSGlue}. In the case of small subextremal masses, we can say more about the domain of existence of $g$; we discuss this at the end of~\S\ref{SsSC}.

\begin{figure}[!ht]
\centering
\includegraphics{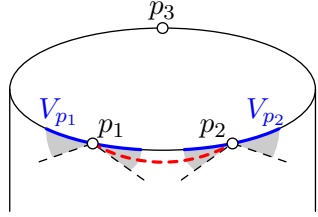}
\caption{Illustration of Theorem~\ref{ThmSGlue}, focusing on a neighborhood of $p_1,p_2$; the shaded regions indicate where we glue in the SdS metrics $g_{p_i,\bhm_i}$, $i=1,2$. The blue segments indicate the sets $V_{p_i}$. The red dashed line indicates a piece of the past boundary of the domain $U$ on which we construct the metric $g$. The difference to Figure~\ref{FigIBaby} is that here we do not require the masses to be subextremal, hence we content ourselves with gluing the far end of the cosmological region of several SdS black holes into de~Sitter space.}
\label{FigSGlue}
\end{figure}

\begin{rmk}
\label{RmkSBH}
  We only explicitly describe here how to glue a piece of the cosmological region of an SdS black hole into de~Sitter space. As is well-known (see e.g.\ \cite[\S3.1]{HintzVasyKdSStability}), the metric $g_\bhm$ in~\eqref{EqSMetric} merely has a coordinate singularity at the \emph{cosmological horizon} $r=r_+$ if the mass is subextremal, meaning $0<9\Lambda\bhm^2<1$. After a suitable (singular) coordinate change, $g_\bhm$ is analytic. There is another coordinate singularity at the event horizon, located at the second largest root of $\Lambda r^2/3-1+2\bhm/r$, beyond which the metric can again be extended analytically. Thus, one can paste these extended subextremal SdS metrics into neighborhoods of $p_i$ and thus, via Theorem~\ref{ThmSGlue}, glue subextremal SdS metrics, extended as far as one wishes, into de~Sitter space. This is depicted in Figure~\ref{FigIBaby}.
\end{rmk}

We show the necessity of the balance assumption under certain decay assumptions on $g$:

\begin{thm}
\label{ThmSUniq}
  Let $(p_1,\bhm_1),\ldots,(p_N,\bhm_N)\in\Sph^3\times\R$, with the $p_i$ pairwise distinct. Suppose $g$ satisfies the conclusions~\eqref{ItSGlueEin}--\eqref{ItSGlueSdS} of Theorem~\usref{ThmSGlue}. If, for some $\eps>0$, we have $g-g_\dS\in\tau^3(\log\tau)\CI+\tau^3\CI+\cA^{3+\eps}(M;S^2\,{}^0 T^*M)$, then $\{(p_1,\bhm_1),\ldots,(p_N,\bhm_N)\}$ is balanced.
\end{thm}

In particular, this applies for metrics $g$ satisfying part~\eqref{ItSGluePert} of Theorem~\ref{ThmSGlue}. The uniqueness theorem is not sharp; the inclusion of a $\tau^3\log\tau$ term merely serves as a demonstration that the inclusion of a logarithmic term does not help in constructing a (formal) solution of $P_0(g)=0$ when the balance condition is violated. The determination of \emph{sharp} conditions under which the balance condition of Definition~\ref{DefSBalance} is necessary for the existence of a metric $g$ satisfying~\eqref{ItSGlueEin}--\eqref{ItSGlueSdS} is left as an open problem. We remark that the analysis of the Einstein vacuum equation for metrics $g$ with $g-g_\dS\in\tau^2\CI$ or $\cA^2$ (or even less decay) is necessarily \emph{nonlinear} on the level of $\cO(\tau^4)$ contributions to $\Ric(g)-\Lambda g$.

Let $\chi_i\in\CI(\pa M)$ denote cutoffs, identically $1$ near $\ol{V_{p_i}}$, and with mutually disjoint supports; put $\chi_0:=1-\sum_{i=1}^N\chi_i$. The starting point of the proof of Theorem~\ref{ThmSGlue} is the naively glued metric
\begin{equation}
\label{EqSOGlued}
  g_0 := \chi_0 g_\dS + \sum_{i=1}^N \chi_i g_{p_i,\bhm_i}.
\end{equation}
Away from the points $p_i$, we have $g_0-g_\dS\in\tau^3\CI$. We shall show in~\S\ref{SsSO} that the failure $P_0(g_0)=2(\Ric(g_0)-\Lambda g_0)$ of $g_0$ to solve the Einstein vacuum equation lies in $\tau^4\CI$ and is supported away from the $p_i$, but it is always nonzero except in the trivial case that $\bhm_i=0$ for all $i$. The goal is to find a correction $\tilde g\in\tau^3\CI$, with support disjoint from $\ol{V_{p_i}}$, such that $P_0(g_0+\tilde g)=0$. To accomplish this, we proceed in several steps:

\begin{enumerate}
\item We improve the error to $P_0(g_0+\tilde g_0)\in\tau^5\CI$ by solving an underdetermined divergence equation for $\tilde g_0$; the balance condition ensures the solvability, while the underdetermined nature of the equation enables us to choose $\tilde g_0$ to vanish identically near the $p_i$. See~\S\ref{SsSO}.
\item We find $g$ in the wave map gauge with background metric $g^0:=g_0+\tilde g_0$ by solving a suitable gauge-fixed Einstein equation $P(g)=0$. This is done in two steps.
  \begin{enumerate}
  \item One can construct $g_1=g^0+\tilde g_1$, $\tilde g_1\in\tau^5\CI$, with $P(g_1)\in\CIdot(M)$ by inverting the indicial family of $D_{g_\dS}P$ and using a Borel summation argument. See~\S\ref{SsSG}.
  \item In order to solve away the final error, we solve the quasilinear wave equation $P(g_1+\tilde g_2)=0$ backwards from $\pa M$, with solution $\tilde g_1\in\CIdot(M)$. See~\S\ref{SsSC}.
  \end{enumerate}
\item Also in~\S\ref{SsSC}, we show that $g$ solves the Einstein vacuum equation by means of the usual argument involving the second Bianchi identity and a unique continuation argument at $\pa M$.
\end{enumerate}

With $P_0$ as in Definition~\ref{Def0EEinOp}, we shall write from now on
\begin{equation}
\label{EqSLinOp}
  L_0 := L_{g_\dS,0} = D_{g_\dS}P_0.
\end{equation}

\subsection{Obstructed problem for the leading order correction}
\label{SsSO}

We will prove:

\begin{prop}
\label{PropSO}
  Under the assumptions of Theorem~\usref{ThmSGlue}, and with $g_0$ defined by~\eqref{EqSOGlued}, there exists $\tilde g_0\in\tau^3\CI(M;S^2\,{}^0 T^*M)$, vanishing near $\bigcup_i\ol{V_{p_i}}$, so that $P_0(g_0+\tilde g_0)\in\tau^5\CI$.
\end{prop}

We begin by computing the error produced by naively gluing a single SdS black hole into a neighborhood of $p_0\in\pa M$:

\begin{lemma}
\label{LemmaSO1}
  In the coordinates~\eqref{EqSMetric}, let $\chi\in\CI(\R_t)$ be identically $1$ for large $t$, and put $g_0=\chi(t) g_\bhm+(1-\chi(t))g_\dS$. With $P_0=2(\Ric-\Lambda)$ as in~\eqref{Eq0EEinOp}, we then have
  \[
    P_0(g_0) \equiv \tau_s^4\Err_{s,0} \bmod \tau_s^5\CI,\quad
    \Err_{s,0}=2\frac{d\tau_s}{\tau_s} \otimes_s \frac{12\bhm}{\Lambda}\frac{d\chi}{\tau_s}.
  \]
\end{lemma}
\begin{proof}
  Since $g_\bhm$ and $g_\dS$ solve the Einstein equation, $P_0(g_0)$ is supported on $\{\chi\neq 0,1\}$. By Lemma~\ref{Lemma0EEin} and in view of~\eqref{EqSODiff}, we have
  \begin{equation}
  \label{EqSO1Nonlin}
    P_0(g_0) = P_0\bigl(g_\dS + \chi(g_\bhm-g_\dS)\bigr) \equiv L_0\bigl(\chi(g_\bhm-g_\dS)\bigr) \bmod \tau_s^6\CI;
  \end{equation}
  but for $\chi\equiv 1$, the left hand side vanishes, hence
  \begin{equation}
  \label{EqSO1Lin}
    L_0(g_\bhm-g_\dS)\in\tau_s^6\CI.
  \end{equation}
  Now, note that $\gamma:=g_\bhm-g_\dS=\tau_s^3\gamma_3+\tau_s^4\gamma_4+\cO(\tau_s^5)$, with $\gamma_3$ and $\gamma_4$ independent of $\tau_s$ when expressed in terms of the bundle splitting~\eqref{Eq0ESplitTr}, with $d\tau_s/\tau_s=-d r/r$, $\tau_s$, and $h_s=(\Lambda^2/9)d t^2+(\Lambda/3)g_{\Sph^2}$ (see~\eqref{Eq0BdSStaticBdy}) taking the roles of $e^0$, $\tau$, and $h$; explicitly,
  \[
    \gamma_3 = \frac{18\bhm}{\Lambda^2}\frac{d r^2}{r^2} + 2\bhm(r\,d t)^2 = \Bigl(\frac{18\bhm}{\Lambda^2}, 0, \frac{6\bhm}{\Lambda^2}, \frac{4\bhm}{3}d t^2-\frac{2\bhm}{\Lambda}g_{\Sph^2} \Bigr).
  \]
  In view of~\eqref{EqSO1Lin}, or by direct calculation using Corollary~\ref{Cor0EInd}, we have $I(L_0,3)\gamma_3=0$ and $I(L_0[\tau_s],3)\gamma_3+I(L_0,4)\gamma_4=0$; thus,~\eqref{EqSO1Nonlin} implies that, modulo $\tau_s^5\CI$,
  \begin{align}
    P_0(g_0) &\equiv \tau_s^4\bigl(I(L_0[\tau_s],3)(\chi\gamma_3) + I(L_0,4)(\chi\gamma_4)\bigr) \nonumber\\
      &= \tau_s^4 \bigl[ I(L_0[\tau_s],3),\chi \bigr] \gamma_3 \nonumber\\
   \label{EqSO1Calc}
      &= \tau_s^4(\Lambda/3)\bigl( 0, [-3\delta_{h_s},\chi](\gamma_3)_{T T 0}, 0, 0 \bigr) \\
      &= \tau_s^4 \Bigl(0, \frac{12\bhm}{\Lambda}d\chi, 0, 0 \Bigr) \nonumber
  \end{align}
  since $[\delta_{h_s},\chi]=-\iota_{\nabla^{h_s}\chi}$, $\iota$ denoting contraction, and $\nabla^{h_s}\chi=\chi'\nabla^{h_s}t=9\Lambda^{-2}\chi'(t)\pa_t$.
\end{proof}

Thus, $\Err_{s,0}=\cO(\tau_s^4)$ is a tangential-normal tensor. In order to proceed, let us pretend we want to glue a single SdS black hole into $M$. Since $I(L_0,4)_{N T}=0$ by Corollary~\ref{Cor0EInd}, we \emph{cannot} solve away $\Err_{s,0}$ with a $\cO(\tau_s^4)$ metric correction. Since $\Err_{s,0}=\cO(\tau_s^4)$ is due to the $\cO(\tau_s^3)$ difference of the metrics $g_\bhm$ and $g_\dS$, we shall instead attempt to solve away $\Err_{s,0}$ with a $\cO(\tau_s^3)$ correction \emph{with support not containing $p_0$}. To this end, note first that by Corollary~\ref{Cor0EInd},
\begin{equation}
\label{EqSInd3}
  \ker I(L_0,3) = \ker\tr_{h_s}{} \oplus \R\bigl(3(e^0)^2+h_s\bigr) \oplus \bigl(2 e^0\otimes_s\tau_s^{-1}T^*X\bigr).
\end{equation}
Written as a block matrix with respect to this splitting and~\eqref{Eq0ESplitTr}, we furthermore have
\begin{equation}
\label{EqSIndlot3}
  3\Lambda^{-1} I(L_0[\tau_s],3) = \begin{pmatrix} 0 & 0 & -4\delta_{h_s} \\ -3\delta_{h_s} & 0 & 0 \\ 0 & 0 & -\frac{4}{3}\delta_{h_s} \\ 0 & 0 & -2\delta_{h_s,0}^* \end{pmatrix} \colon \CI(M;\ker I(L_0,3))\to\CI(M;S^2\,{}^0 T^*M).
\end{equation}
Thus, we need to find $k\in\CI(\pa M;\ker\tr_{h_s})$ which vanishes near $p_0$ and which solves $-(\Lambda/3)3\delta_{h_s} k=(\Err_{s,0})_{N T}=12\bhm\Lambda^{-1}d\chi$. A necessary condition for solvability is that the right hand side be $L^2(\pa M;|d h_s|)$-orthogonal to the space $\ker\delta_{h_s,0}^*\in\CI(\pa M;T^*\pa M)$ of conformal Killing 1-forms. Identifying 1-forms with vector fields via the metric $h_s$, this condition reads
\begin{equation}
\label{EqSOIndSolvCrit}
  \int_{\pa M} V(\Err_{s,0})_{N T}\,|d h_s|=0\quad\text{for all conformal Killing vector fields $V$ on $(\pa M,h_s)$.}
\end{equation}

The space of conformal Killing vector fields only depends on the conformal class of the metric.\footnote{Indeed, if $(X,h)$ is a Riemannian manifold and $V$ is a conformal Killing vector fields, so $\cL_V h=f h$ for some $f\in\CI(X)$, then $\cL_V(e^{2\varphi}h)=e^{2\varphi}(f+2 V\varphi)h$ for any $\varphi\in\CI(X)$.} Note then that $h_s$ is conformal to $g_{\Sph^3}$; indeed, $h_s=\tau_s^2\tau^{-2}g_{\Sph^3}$. The conformal Killing vector fields of the standard $n$-sphere, $n\geq 3$, are well-known (see e.g.\ \cite[\S1.4]{SchottenloherCFT} and use the stereographic projection):

\begin{prop}
\label{PropSConfKillSph}
  The space $\mathfrak{conf}_n=\ker\delta_{g_{\Sph^n},0}^*\subset\cV(\Sph^n)$ is a direct sum
  \[
    \mathfrak{conf}_n = \mathfrak{iso}_n \oplus \mathfrak{scal}_n,
  \]
  where $\mathfrak{iso}_n=\ker\delta_{g_{\Sph^n}}^*\cong\mathfrak{so}_{n+1}$ is the space of Killing vector fields (rotations) on $\Sph^n$, and
  \begin{equation}
  \label{EqSConfKillSph}
    \mathfrak{scal}_n = \{ S_q \colon q\in\R^n \},\quad S_q \colon \Sph^n \ni p \mapsto q-\la q,p\ra p \in T_p\Sph^n,
  \end{equation}
  where $\la\cdot,\cdot\ra$ is the standard inner product on $\R^{n+1}$.
\end{prop}

Passing from $\tau_s$ to the global boundary defining function $\tau$, the error $\Err_0$ in $P_0(g_0)\equiv\tau^4\Err_0\bmod\tau^5\CI$ has normal-tangential component
\begin{equation}
\label{EqSOConf}
  (\Err_0)_{N T} = \tau^{-4} P_0(g_0)(\tau\pa_\tau,\tau\cdot)|_{T\Sph^3} = \tau^{-3}\tau_s^3(\Err_{s,0})_{N T} = \tau^{-3}\tau_s^3\frac{12\bhm}{\Lambda}d\chi.
\end{equation}

\begin{rmk}
\label{RmkSConfInv}
  Since $|d g_{\Sph^3}|=\tau^3\tau_s^{-3}|d h_s|$, the solvability condition~\eqref{EqSOIndSolvCrit} is equivalent to
  \[
    \int_{\pa M} V(\Err_0)_{N T} |d g_{\Sph^3}| = 0\quad\forall\,V\in\mathfrak{conf}_3.
  \]
  This has the same form as~\eqref{EqSOIndSolvCrit}; thus, the condition~\eqref{EqSOIndSolvCrit} is conformally invariant.
\end{rmk}

Now, at $\tau=0$, both $\tau_s/\tau$ and $\chi$ are functions of $t$ only, thus of $\tilde R=|\tilde x|$ by~\eqref{Eq0BdSStatic}, and thus of the geodesic distance $d_{\Sph^3}(p_0,-)$ from the point $p_0\in\Sph^3$ by~\eqref{Eq0BdSNoncpt}; thus, we have $\tau^{-3}\tau_s^3 d\chi=d\tilde\chi$ for some $\tilde\chi=\tilde\chi(d_{\Sph^3}(p_0,-))$. With $(\Err_0)_{N T}=\frac{12\bhm}{\Lambda}d\tilde\chi$, we can now compute
\begin{equation}
\label{EqSOIndSolvCrit2}
  \ell_{p_0,\bhm} \in (\mathfrak{conf}_3)^*,\quad \mathfrak{conf}_3 \ni V \mapsto \int_{\Sph^3} V(\Err_0)_{N T}|d g_{\Sph^3}|.
\end{equation}
Namely, for $V\in\mathfrak{iso}_n=\ker\delta^*_{g_{\Sph^3}}\subset\ker\delta_{g_{\Sph^3}}$, integration by parts gives $\ell_{p_0,\bhm}(V)=0$. On the other hand, if $q=p_0\in\Sph^3\subset\R^4$, the vector field $S_q$ is the radial vector field pointing towards $p_0$, and
\[
  \ell_{p_0,\bhm}(S_{p_0})=C_0\cdot 12\bhm/\Lambda,
\]
where $C_0\in\R$ is a universal constant. We claim that $C_0\neq 0$. Indeed, passing back to~\eqref{EqSOIndSolvCrit} note that $S_q$ is a radial vector field, i.e.\ a $\CI(\R_{\tilde R})$-multiple of $\pa_{\tilde R}$ and thus a $\CI(\R_t)$-multiple of $\pa_t$, and hence must be a constant nonzero multiple of $\pa_t$, which is the unique conformal Killing vector field of $h_s$ of this form. (In fact, $\pa_t$ is Killing for $h_s$.) But 
\[
  \int \pa_t\Bigl(\frac{12\bhm}{\Lambda}\chi'(t)d t\Bigr) |d h_s| = \frac{12\bhm}{\Lambda}\cdot(0-1)\cdot\frac{\Lambda^2}{9}\vol(\Sph^2) = -\frac{16\pi\Lambda\bhm}{3}
\]
is nonzero, proving that $C_0\neq 0$.

Finally, if $q\perp p_0$, then the integrand $S_q(\Err_0)_{N T}$ in~\eqref{EqSOIndSolvCrit2} is odd with respect to the reflection across the axis $\R p_0$, hence $\ell_{p_0,\bhm}(S_q)=0$ in this case. Therefore,
\begin{equation}
\label{EqSOIndSolvCrit3}
  \ell_{p_0,\bhm}(S_q) = C_0\frac{12\bhm}{\Lambda}\la p_0,q\ra,\quad q\in\R^4.
\end{equation}
In particular, there is a nontrivial obstruction to gluing a single nontrivial ($\bhm\neq 0$) SdS black hole into $M$. We summarize our findings in the following lemma:

\begin{lemma}
\label{LemmaSO2}
  Given $p_i\in\Sph^3,\bhm_i\in\R$, and cutoff functions $\chi_i$, identically $1$ near $p_i$ and vanishing near $-p_i$, for $i=1,\ldots,N$, set $\Err_{N T}:=\sum_{i=1}^N(\Err_{p_i,\bhm_i})_{N T}\in\CI(\Sph^3,T^*\Sph^3)$, where
  \[
    (\Err_{p_i,\bhm_i})_{N T}:= \tau^{-4}P_0\bigl(\chi_i g_{p_i,\bhm_i} + (1-\chi_i)g_\dS\bigr)(\tau\pa_\tau,\tau W)|_{\tau=0},\quad W\in T\Sph^3.
  \]
  Then we have
  \[
    \int_{\Sph^3} V(\Err_{N T}) |d g_{\Sph^3}| = 0\quad\forall\,V\in\mathfrak{conf}_3
  \]
  if and only if $\{(p_1,\bhm_1),\ldots,(p_N,\bhm_N)\}$ is balanced as in Definition~\usref{DefSBalance}.
\end{lemma}
\begin{proof}
  If every $\chi_i$ is a radial cutoff, relative to the point $p_i$, the claim follows from~\eqref{EqSOIndSolvCrit3} and the fact that $\sum_{i=1}^N\la \bhm_i p_i,q\ra=0$ for all $q\in\R^4$ if and only if $\sum_{i=1}^N\bhm_i p_i=0$, which is precisely the balance condition.

  It remains to prove the lemma for general cutoffs. Observe that the difference of error terms produced by two different cutoffs $\chi_{i,j}$, $j=1,2$, to a neighborhood of the same point $p_i$ lies in the range of $\delta_h$ acting on smooth 1-forms supported away from $p_i$. Indeed, similarly to~\eqref{EqSO1Calc}, the difference is equal to $\bigl(0,-(\Lambda/3)3\delta_h\bigl((\chi_{i,1}-\chi_{i,2})(\gamma_3)_{T T 0}\bigr),0,0\bigr)$ (in the splitting~\eqref{Eq0ESplitTr}) where $(\gamma_3)_{T T 0}$ is the trace-free part of the tangential-tangential component (with respect to~\eqref{EqSMfd}) of $g_{p_i,\bhm_i}-g_\dS$; note that $\chi_{i,1}-\chi_{i,2}$ vanishes near $p_i$.
\end{proof}

Since $\delta_h$ acting on trace-free symmetric 2-tensors has surjective principal symbol, standard elliptic theory implies that under the balance condition, the error of Lemma~\ref{LemmaSO2} can be written as $\Err_{N T}=\delta_h k$ for some $k\in\CI(\Sph^3;S^2 T^*\Sph^3)$, $\tr_{g_{\Sph^3}}k=0$. Crucially, we can do much better, since the overdetermined nature of this equation allows us to find $k$ with strong support restrictions due to the following result due to Delay:

\begin{thm}
\label{ThmSODelay}
  (Delay \cite[Theorem~1.3, Proposition~9.7, Corollary~8.4]{DelayCompact}.) Let $(X,h)$ be a smooth Riemannian manifold, and let $\Omega\subset X$ be open. Suppose $f\in\CI(X;T^*X)$ satisfies $\supp f\Subset\Omega$ and $\int_\Omega V(f)|d h|=0$ for all $V\in\cV(\Omega)$ satisfying $\delta_{h,0}^*V^\flat=0$.\footnote{Note that if $\Omega$ has several connected components, the space of such $V$ is larger than the space of conformal Killing vector fields on $X$.} Then there exists $k\in\CI(X;S^2 T^*X)$ with $\tr_h k=0$ and $\supp k\subset\bar\Omega$ such that $\delta_h k=f$.
\end{thm}

\begin{proof}[Proof of Proposition~\usref{PropSO}]
  Define the glued metric $g_0$ as in~\eqref{EqSOGlued}. As in (the proof of) Lemma~\ref{LemmaSO1}, we define $\Err$ to be the $\tau^4$ coefficient of
  \[
    P_0(g_0)\equiv\sum_{i=1}^N L_0(g_\dS+\chi_i(g_{p_i,\bhm_i}-g_\dS))\bmod\tau^5\CI(M;S^2\,{}^0 T^*M).
  \]
  We can thus compute $\Err$ using Lemma~\ref{LemmaSO1}; its normal-tangential component is equal to $\Err_{N T}$ as defined in Lemma~\ref{LemmaSO2}.
  
  Since the cutoffs $\chi_i$ are identically $1$ in a neighborhood of $\ol{V_{p_i}}$, there exists an open set $\Omega\subset\Sph^3$ with $\bar\Omega\cap\ol{V_{p_i}}=\emptyset$ for all $i$, and so that $\supp\Err_{N T}\Subset\Omega$; moreover, we may choose $\Omega$ to be connected. Suppose $V\in\cV(\Omega)$ is a conformal Killing vector field. We contend that $V=\tilde V|_\Omega$ for a conformal Killing vector field $\tilde V\in\cV(\Sph^3)$. Indeed, on any connected $n$-dimensional Riemannian manifold, the dimension of the space of conformal Killing vector fields is at most $(n+1)(n+2)/2$, and on $\Sph^n$ it is equal to this. We conclude that the restriction map $\ker_{\cV(\Sph^3)}\delta_{h,0}^*\to\ker_{\cV(\Omega)}\delta_{h,0}^*$, which is injective (as a consequence of the explicit description in Proposition~\ref{PropSConfKillSph}), must be an isomorphism.

  By Lemma~\ref{LemmaSO2}, the balance condition implies that the conditions of Theorem~\ref{ThmSODelay} are satisfied; thus, there exists $k\in\CI(\Sph^3;S^2 T^*\Sph^3)$, $\tr_{g_{\Sph^3}}k=0$, $\supp k\subset\bar\Omega$, with
  \begin{equation}
  \label{EqSODivEq}
    -(\Lambda/3)3\delta_h k=-\Err_{N T}.
  \end{equation}
  In the splitting~\eqref{Eq0ESplitTr}, put
  \[
    \tilde g_0 = (0, 0, 0, k) \in \CI(M;S^2\,{}^0 T^*M).
  \]
  In view of~\eqref{EqSInd3}, we have $\tilde g_0\in\ker I(L_0,3)$. Therefore, Lemma~\ref{Lemma0EEin} and Corollary~\ref{Cor0EInd} imply that, modulo $\tau^5\CI(M;S^2\,{}^0 T^*M)$,
  \[
    P_0(g_0+\tilde g_0) \equiv \tau^4\bigl( (0,\Err_{N T},0,0) + I(L_0[\tau],3)\tilde g_0 \bigr) \equiv 0,
  \]
  finishing the proof.
\end{proof}

\begin{rmk}
\label{RmkSOWeyl}
  A direct calculation shows that the error $\Err_{N T}$ is, up to a constant rescaling, equal to the divergence (with respect to the induced metric $h$ on $\pa M$) of the leading order term of the normal-tangential-normal-tangential part of the Weyl tensor of $g_0$. Thus, \cite[Lemma~(3.1)]{FriedrichDeSitterPastSimple}, in particular \cite[Equation~(3.12)]{FriedrichDeSitterPastSimple}, requires the solution of the same divergence equation~\eqref{EqSODivEq}. Solving Friedrich's conformal Einstein field equations then produces a solution of the Einstein vacuum equation and proves Theorem~\ref{ThmSGlue}. As motivated in the introduction, we give a different, self-contained proof below.
\end{rmk}

\subsection{Gauge fixing; construction of the formal solution}
\label{SsSG}

Following DeTurck~\cite{DeTurckPrescribedRicci}, we make the following definition:

\begin{definition}
\label{DefSG}
  Let $g^0$ and $g$ denote two Lorentzian metrics on the same manifold.
  \begin{enumerate}
  \item We define the \emph{gauge 1-form} by
    \[
      \Ups(g;g^0) = g(g^0)^{-1}\delta_g\sfG_g g^0.
    \]
  \item The \emph{gauge-fixed Einstein operator} is
    \[
      P(g;g^0) := 2\bigl(\Ric(g)-\Lambda - \delta_g^*\Ups(g;g^0)\bigr).
    \]
    Its linearization in the first argument is denoted
    \[
      L_{g,g^0}(\dot g) = D_1|_g P(\dot g;g^0) := \frac{d}{d s}P(g+s\dot g;g^0)|_{s=0}.
    \]
  \end{enumerate}
\end{definition}

We first discuss general properties of these operators. Following~\cite[\S3]{GrahamLeeConformalEinstein}, we have, using the Levi-Civita connection of $g$,
\begin{align*}
  &D_1|_g\Ups(\dot g;g^0) = -\delta_g\sfG_g\dot g + \sC(\dot g) - \sD(\dot g), \\
  &\qquad C_{\mu\nu}^\kappa = \half((g^0)^{-1})^{\kappa\lambda}(g^0_{\mu\lambda;\nu}+g^0_{\nu\lambda;\mu}-g^0_{\mu\nu;\lambda}),\quad
    D^\kappa=g^{\mu\nu}C_{\mu\nu}^\kappa, \\
  &\qquad \sC(\dot g)_\kappa = g_{\kappa\lambda}C_{\mu\nu}^\lambda\dot g^{\mu\nu},\quad
    \sD(\dot g)_\kappa = D^\lambda\dot g_{\kappa\lambda}.
\end{align*}
In the special case $g=g^0$, we have $\sC\equiv 0$ and $\sD\equiv 0$, and therefore by~\eqref{Eq0ELinRic}
\[
  L_{g,g} = \Box_g + 2\sR_g - 2\Lambda.
\]
If moreover $g=(3/\Lambda)\tau^{-2}(-d\tau^2+h(x,d x))$ is a product metric as in~\eqref{Eq0EMetricProd}, then in the splitting~\eqref{Eq0ESplitTr}
\begin{subequations}
\begin{equation}
\label{EqSGIndOp}
  3\Lambda^{-1} I(L_{g,g},\lambda) = \lambda^2-3\lambda + \begin{pmatrix} -6 & 0 & 0 & 0 \\ 0 & -4 & 0 & 0 \\ 0 & 0 & -6 & 0 \\ 0 & 0 & 0 & 0 \end{pmatrix}.
\end{equation}
For later use, we note that the indicial roots are, in increasing order,
\begin{equation}
\label{EqSGIndRoots}
  \half(3-\sqrt{33})\in(-2,-1),\quad
  -1,\quad
  4,\quad
  \half(3+\sqrt{33})\in(4,5).
\end{equation}
\end{subequations}

For more general metrics, arguments similar to those in Lemmas~\ref{Lemma0EEinLinPert}--\ref{Lemma0EEin} give:
\begin{lemma}
\label{LemmaSGLin}
  Let $g_{0 0}\in\CI(M;S^2\,{}^0 T^*M)$, and suppose that $g^0\in\CI(M;S^2\,{}^0 T^*M)$ is such that $g^0-g_{0 0}\in\tau^m\CI$ for some $m\in\N$. Suppose moreover that $\tilde g\in\tau^{m'}\CI$ for some $m'\in\N$, $m'\geq m$, and put $g=g^0+\tilde g$.\footnote{In particular, $g-g_{0 0}\in\tau^m\CI$.} Then $L_{g,g^0}-L_{g_{0 0},g_{0 0}}\in\tau^m\Diff_0^2(M;S^2\,{}^0 T^*M)$. Moreover, if $\dot g\in\tau^{m_2}\CI$, then $P(g+\dot g;g^0) - P(g;g^0) - L_{g;g^0}\dot g \in \tau^{2 m_2}\CI$.
\end{lemma}

Applying this with $g_{0 0}$ a product metric as in~\eqref{Eq0EMetricProd}, we conclude that $I(L_{g,g^0},\lambda)$ is equal to the right hand side of~\eqref{EqSGIndOp}.

Returning to the black hole gluing problem and the notation of Proposition~\ref{PropSO}, we now define the `background metric' $g^0$ to be
\begin{equation}
\label{EqSGBackground}
  g^0 = g_0+\tilde g_0 \in \CI(M;S^2\,{}^0 T^*M).
\end{equation}
Since $\Ups(g^0;g^0)=0$ and $\Ric(g_0)-\Lambda g_0\in\tau^5\CI$, we have
\[
  P(g^0;g^0) \in \tau^5\CI(M;S^2\,{}^0 T^*M);
\]
moreover, by construction, $P(g^0;g^0)$ vanishes near $\bigcup_i\ol{V_{p_i}}$.

\begin{prop}
\label{PropSGFormal}
  Under the assumptions of Proposition~\usref{PropSO}, and with $g^0$ as in~\eqref{EqSGBackground}, there exists a metric perturbation $\tilde g_1\in\tau^5\CI(M;S^2\,{}^0 T^*M)$, vanishing near $\bigcup_i\ol{V_{p_i}}$, so that $P(g^0+\tilde g_1;g^0)\in\tau^\infty\CI=\bigcap_{m\in\N}\tau^m\CI$ (i.e.\ vanishing to infinite order at $\tau=0$).
\end{prop}
\begin{proof}
  Suppose we have already found $\tilde g_1$ as in the statement and with $P(g^0+\tilde g_1;g^0)\in\tau^m\CI$ for some $m\geq 5$; note that for $m=5$, this holds for $\tilde g_1=0$. Moreover, under these assumptions, $P(g^0+\tilde g;g^0)$ vanishes near $\bigcup_i\ol{V_{p_i}}$. Then, for $\dot g=\tau^m\dot g_0\in\tau^m\CI$, we have, using Lemma~\ref{LemmaSGLin} and noting that $g^0-g_\dS\in\tau^3\CI$,
  \begin{align*}
    P(g^0+\tilde g_1+\dot g;g^0) &\equiv P(g^0+\tilde g_1;g^0) + L_{g^0+\tilde g_1;g^0}\dot g \bmod \tau^{2 m}\CI \\
      &\equiv P(g^0+\tilde g_1;g^0) + L_{g_\dS,g_\dS}\dot g \bmod \tau^{m+3}\CI \\
      &\equiv P(g^0+\tilde g_1;g^0) + \tau^m I(L_{g_\dS,g_\dS},m)\dot g_0 \bmod \tau^{m+1}\CI.
  \end{align*}
  But for $m\geq 5$, $I(L_{g_\dS,g_\dS},m)$ is invertible pointwise on $\pa M$ in view of~\eqref{EqSGIndRoots}, hence we can find $\dot g_0\in\CI(\pa M;S^2\,{}^0 T^*_{\pa M}M)$, vanishing near $\bigcup_i\ol{V_{p_i}}$ such that this vanishes (modulo $\tau^{m+1}\CI$). Replacing $\tilde g_1$ by $\tilde g_1+\dot g$ improves the order of vanishing of $P(g^0+\tilde g_1;g^0)$ at $\tau=0$ by one order. A Borel summation argument produces a formal solution $\tilde g_1\in\tau^5\CI$.
\end{proof}

\subsection{Solving the nonlinear equation; conclusion of the construction}
\label{SsSC}

Using indicial operator arguments, one cannot go beyond Proposition~\ref{PropSGFormal}; the remaining (`trivial') error can however easily be solved away:
\begin{prop}
\label{PropSCNonlinear}
  With $g_1:=g^0+\tilde g_1$ defined using Proposition~\usref{PropSGFormal}, there exists $\tilde g_2\in\tau^\infty\CI(M;S^2\,{}^0 T^*M)$, vanishing near $\bigcup_i\ol{V_{p_i}}$, so that $g:=g_1+\tilde g_2$ satisfies
  \begin{equation}
  \label{EqSCEq}
    P(g;g^0)=0\ \text{near}\ \tau=0.
  \end{equation}
\end{prop}
\begin{proof}
  The key point is that forced linear wave equations on de~Sitter space, or with any product metric of the form~\eqref{Eq0EMetricProd} or indeed any metric smoothly asymptotic to it, can be solved backwards on function spaces encoding sufficient polynomial decay in $\tau$ (i.e.\ sufficient \emph{exponential} decay in $-\log\tau$), with the solution unique in such spaces; see the proof of \cite[Lemma~1]{ZworskiRevisitVasy} (where $N$ is the order of decay in $|x_1|$, $x_1:=-\tau^2$) for the relevant energy estimate, and the beginning of \cite[\S3]{VasyWaveOndS} (where our $\tau,x$ are denoted $x,y$).
  
  Since the error we need to solve away vanishes to all orders at $\tau=0$, there is no need to choose vector field multipliers and \emph{positive definite} vector bundle inner products on $S^2\,{}^0 T^*M$ carefully in such energy estimates; rather, fixing \emph{any} smooth positive definite inner product on $S^2\,{}^0 T^*M$, one obtains an energy estimate using the vector field multiplier $\tau^{-2 N}\tau\pa_\tau$ when $N$ is sufficiently large. Indeed, the only contribution to the bulk term in the estimate which comes with a factor $N$ in front arises from differentiating $\tau^{-2 N}$ and is thus of the form $-2 N\tau^{-2 N} E(\tau\pa_\tau,\tau\pa_\tau)$, where $E$ is the energy-momentum tensor of the wave $u$ one wishes to estimate; all other bulk terms can be estimated by $\tau^{-2 N}$ times a \emph{bounded} (independently of $N$) multiple of $|u|^2+|\tau\pa_\tau u|^2+|\tau\pa_x u|^2$. But since $\tau\pa_\tau$ is timelike, choosing $N$ large enough produces a coercive bulk term, and one obtains, for example, an estimate $\|u\|_{\tau^N\Hext_0^1(\Omega)}\leq C\|\Box_{g_1}u\|_{\tau^N\Hext_0^0(\Omega)}$ for sufficiently large $N$, where $\Omega=\tau^{-1}([0,\half))$. One can also commute any fixed number of 0-derivatives through the equation and thus (upon increasing $N$ and $C$) obtain the estimate $\|u\|_{\tau^N\Hext_0^{k+1}(\Omega)}\leq C\|\Box_{g_1}u\|_{\tau^N\Hext_0^k(\Omega)}$.

  For the quasilinear wave equation at hand, we work with 0-Sobolev spaces with more than $\half(\dim M)+2=4$ derivatives; thus, fix $k_0=5>4$. Then by a simple adaptation of the standard iteration scheme for solving quasilinear wave equations (see e.g.\ \cite[\S16]{TaylorPDE3}), we obtain, for sufficiently large $N_0$, a solution $\tilde g_2\in\tau^{N_0}\Hext_0^{k_0}(\Omega_0)$ (unique in this space) of equation~\eqref{EqSCEq}, where $\Omega_0=\tau^{-1}([0,\eps_0))$ for sufficiently small $\eps_0>0$. Moreover, $\tilde g_2$ vanishes near each $\ol{V_{p_i}}$ since $P(g_1;g^0)$ does; recall that $g_1$ and $g^0$ are both equal to the Schwarzschild--de~Sitter metric $g_{p_i,\bhm_i}$ near $\ol{V_{p_i}}$.

  For any $k\geq k_0$, one can similarly find a solution of~\eqref{EqSCEq} in the space $\tau^N\Hext_0^k(\Omega)$ where $\Omega$ is a neighborhood of $\tau=0$; since solutions of quasilinear wave equations can be continued (backwards, i.e.\ in the direction of increasing $\tau$) in the same regularity class as long as a \emph{fixed} low regularity norm remains finite, we can in fact take $\Omega=\Omega_0$. We conclude that $\tilde g_2\in \bigcap_{N,k}\tau^N\Hext_0^k(\Omega_0) = \tau^\infty\CI(\Omega_0)$, the final equality following from the fact that $\tau^k\Hext_0^{k+3}(\Omega_0)\subset\cC^k(\Omega_0)$ by Sobolev embedding and using that $\pa_\tau=\tau^{-1}\cdot\tau\pa_\tau$ and $\pa_x=\tau^{-1}\cdot\tau\pa_x$.
\end{proof}

By construction, the metric $g$ meets the requirements~\eqref{ItSGlueSdS}--\eqref{ItSGluePert} of Theorem~\ref{ThmSGlue}. We prove that it also satisfies requirement~\eqref{ItSGlueEin}; recall that $g-g^0\in\tau^5\CI$ by Propositions~\ref{PropSO} and \ref{PropSCNonlinear}, and $g-g^0$ vanishes near $\bigcup\ol{V_{p_i}}$.

\begin{lemma}
\label{LemmaSCUniq}
  Suppose $g,g^0$ are two Lorentzian metrics defined near $\pa M\setminus\{p_1,\ldots,p_N\}$, smooth down to $\tau=0$ as sections of $S^2\,{}^0 T^*M$, and equal, modulo $\tau\CI(M;S^2\,{}^0 T^*M)$, to a metric of product type \eqref{Eq0EMetricProd} near $\pa M\setminus\{p_1,\ldots,p_N\}$. Suppose that near $\bigcup_i\ol{V_{p_i}}$, we have $g=g^0$ and $\Ric(g)-\Lambda g=0$. Suppose moreover that $\Ric(g^0)-\Lambda g^0$ and $g-g^0$ lie in $\tau^5\CI(M;S^2\,{}^0 T^*M)$. If $P(g;g^0)=0$, then $\Ric(g)-\Lambda g=0$ and $\Ups(g;g^0)=0$ near $\pa M$.
\end{lemma}
\begin{proof}
  The conclusion holds trivially near $\ol{V_{p_i}}$. Now, by the second Bianchi identity, the equation $\delta_g\sfG_g P(g;g^0)=0$ implies the wave equation
  \begin{equation}
  \label{EqSCUniqProp}
    2\delta_g\sfG_g\delta_g^*\Ups(g;g^0) = 0
  \end{equation}
  for the gauge 1-form $\Ups(g;g^0)$. By assumption, we have $\Ups(g;g^0)\in\tau^5\CI(M;{}^0 T^*M)$. The idea is to view equation~\eqref{EqSCUniqProp} as a scattering problem (`initial value problem for data at infinity') for $\Ups(g;g^0)$. We need to show that the a priori decay of $\Ups(g;g^0)$ is a suitable replacement for vanishing Cauchy data in the usual proof of short time existence for the Einstein equation, in that it suffices to conclude $\Ups(g;g^0)\equiv 0$.

  We first contend that in fact $\Ups(g;g^0)\in\tau^\infty\CI$ vanishes to infinite order at $\pa M$; this uses an indicial operator argument. Thus, if $g_{0 0}$ is a product metric on a $4$-manifold $M$ as in~\eqref{Eq0EMetricProd}, then we have, in the bundle splitting~\eqref{Eq0ESplitTr},
  \[
    3\Lambda^{-1} I(2\delta_{g_{0 0}}\sfG_{g_{0 0}}\delta_{g_{0 0}}^*,\lambda) = \begin{pmatrix} \lambda^2-3\lambda-6 & 0 \\ 0 & \lambda^2-3\lambda-4 \end{pmatrix};
  \]
  its indicial roots are given by~\eqref{EqSGIndRoots}. If $g-g_{0 0}\in\tau\CI$, then $2\delta_g\sfG_g\delta_g^*-2\delta_{g_{0 0}}\sfG_{g_{0 0}}\delta_{g_{0 0}}^*\in\tau\Diff_0^2$ by arguments similar to Lemmas~\ref{Lemma0EEinLinPert} and \ref{Lemma0EEin}.

  Now, if we already know $\Ups(g;g^0)\in\tau^m\CI$ for some $m\geq 5$ (the case $m=5$ being our starting point), then, writing $\Ups(g;g^0)=\tau^m\Ups_0+\tilde\Ups$ for $\Ups_0\in\CI(\pa M;{}^0 T^*_{\pa M}M)$ ($\tau$-independent) and $\tilde\Ups\in\tau^{m+1}\CI$, equation~\eqref{EqSCUniqProp} implies
  \[
    I(2\delta_g\sfG_g\delta_g^*,m)\Ups_0 = 0.
  \]
  But the indicial operator appearing here is pointwise invertible, hence $\Ups_0\equiv 0$ and therefore $\Ups(g;g^0)=\tilde\Ups\in\tau^{m+1}\CI$. Since $m$ was arbitrary, this proves our contention.

  Finally, the rapid decay of $\Ups(g;g^0)$ at $\tau=0$ (and its vanishing near the $p_i$ where the metric $g$ is singular) implies by a unique continuation argument for the wave equation~\eqref{EqSCUniqProp}, based on an energy estimate with multiplier $\tau^{-2 N}\tau\pa_\tau$ for sufficiently large $N$, that $\Ups(g;g^0)$ vanishes identically near $\tau=0$. See \cite[Lemma~1]{ZworskiRevisitVasy}; a closely related alternative approach is given in \cite[Proposition~5.3]{VasyWaveOndS}.

  Since $P(g;g^0)=0$ and $\Ups(g;g^0)=0$, we conclude that $\Ric(g)-\Lambda g=0$ near $\tau=0$.
\end{proof}

The proof of Theorem~\ref{ThmSGlue} is complete.

\begin{rmk}
\label{RmkSCFormal}
  The Fefferman--Graham construction \cite[\S3]{FeffermanGrahamAmbientBook} produces, given a Riemannian metric $h_0$ and a transverse traceless tensor $h_n$ on an $n$-manifold $\pa M$ a formal solution $g^0$ of $\Ric(g^0)-\Lambda g^0\in\tau^\infty\CI$ on $M=[0,1)_\tau\times\pa M$ of the form $g^0=\tau^{-2}(-d\tau^2+h(\tau,x,d x))$, where $h_0$ and $h_n$ are the coefficients of $\tau^n$ in the polyhomogeneous expansion of $h$. (Concretely, $h$ has an expansion into $\tau^j$, $j\in\N_0$, and $\tau^j\log\tau$ for integer $j\geq n$.) Using this formal solution as a background metric for the gauge-fixed Einstein equation, Proposition~\ref{PropSCNonlinear} and Lemma~\ref{LemmaSCUniq} produce a true solution $g=g^0+\cO(\tau^\infty)$ of $\Ric(g)-\Lambda g=0$. We stress that this does not require any smallness conditions on the data $h_0,h_n$. See \cite[Theorem~1.3]{RodnianskiShlapentokhRothmanSelfSimilar} for a different approach.
\end{rmk}

We end with a discussion of the domain of existence when all masses are subextremal and small; we show that the cosmological horizons of at least two different black holes intersect nontrivially in the maximally globally hyperbolic development of the glued metric $g$ of Theorem~\ref{ThmSGlue} and Remark~\ref{RmkSBH}. Let us work on the partial compactification
\[
  M = (0,\pi/2]_s \times \Sph^3,\quad
  g = (3/\Lambda)\cos^{-2}(s)(-d s^2+g_{\Sph^3}),
\]
of de~Sitter space; the gluing theorem is, so far, local near $s=\pi/2$. Fix $N\geq 2$ distinct points $p_1,\ldots,p_N\in\Sph^3$. If $N=2$, we assume $p_2\neq-p_1$. Denoting by $d$ the Riemannian distance on $(\Sph^3,d_{\Sph^3})$, we set $d_0:=\min_{i\neq j} d(p_i,p_j)\in(0,\pi)$; without loss of generality, suppose the distance is minimized for $p_1,p_2$ so that
\[
  d_0=d(p_1,p_2).
\]
Let moreover $0<r_0<\pi/2$ be less than $\half$ times the smallest radius of any of the balls $V_{p_i}$ in Theorem~\ref{ThmSGlue}. Given subextremal masses $\bhm_1,\ldots,\bhm_N$ so that
\[
  \sD=\{(p_1,\bhm_1),\ldots,(p_N,\bhm_N)\}
\]
is balanced, the metric $g$ constructed in Theorem~\ref{ThmSGlue} is equal to $g_{p_i,\bhm_i}$ in the domain of dependence of $B(p_i,2 r_0)$. Fix $0<\eps<r_0/16$, and define $\Sigma_0\subset M$ as the union of $\Sph^3\setminus\bigcup_{i=1}^N B(p_i,r_0)\subset\pa M$ and the $N$ spacelike surfaces
\[
  \cN_i:=\bigl\{(s,p)\in M\colon d(p,p_i)=r_0-\eta(\pi/2-s),\ \pi/2-s\leq r_0/2+4\eps\bigr\},
\]
where $0\leq\eta-1<1$ is fixed so that $r_0-\eta(r_0/2+4\eps)>0$. Note that $\cN_i$ penetrates the cosmological horizon of an observer in de~Sitter space tending to $p_i$, i.e.\ the backwards light cone from $(\pi/2,p_i)\in\pa M$. We denote by
\[
  \cS_{\eta,i} := \bigl\{(s,p)\in M \colon d(p,p_i)=r_0-\eta(r_0/2+\eps),\ \pi/2-s=r_0/2+\eps\bigr\} \subset \cN_i
\]
a sphere which lies just inside of said cosmological horizon when $\eta-1\ll 1$. See Figure~\ref{FigSLong}.

\begin{figure}[!ht]
\centering
\includegraphics{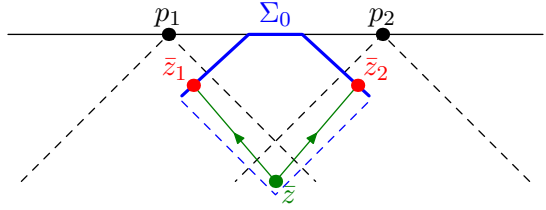}
\caption{Illustration of the argument giving long-time existence of the metric $g$ in Theorem~\ref{ThmSGlue} for small subextremal masses. The geometry shown here is that of de~Sitter space. The region enclosed by the blue lines is the backwards domain of dependence of the spacelike surface $\Sigma_0$. By Cauchy stability, two observers starting at $\bar z$ can reach the points $\bar z_1\in\cS_1$ and $\bar z_2\in\cS_2$ if we glue sufficiently light black holes into $p_1$ and $p_2$, in which case the black dashed lines become the cosmological horizons of the black holes.}
\label{FigSLong}
\end{figure}

Consider the rescaled gluing data
\[
  \lambda\sD := \{(p_1,\lambda\bhm_1),\ldots,(p_N,\lambda\bhm_N)\}
\]
for $\lambda\geq 0$. For $\lambda=0$, all masses vanish, hence we are gluing pieces of de~Sitter space into de~Sitter space---the result of course being de~Sitter space, with metric defined globally on $M$. Let $q\in\Sph^3$ denote the midpoint between $p_1$ and $p_2$ (so $d(p_1,q)=d(p_2,q)=d_0/2$), and let
\[
  \bar z_j = (\bar s_j,\bar p_j)\in\cS_{\eta,j},\quad j=1,2,
\]
denote the point on $\cS_{\eta,j}$ for which $\bar p_j$ is closest to $q$. If we had $\eta=1$, then for any point $z=(s,q)$ with $\pi/2-s>d_0/2+2\eps$, we would have
\[
  d(q,\bar p_j) = \half(d_0-r_0)+\eps < (\pi/2-s) - r_0/2-\eps = (\pi/2-s) - (\pi/2-\bar s_j),
\]
so $(\bar s_1,\bar p_1)$ and $(\bar s_2,\bar p_2)$ are both contained in the timelike future $I^+(z)$ of $z$. For $\eta>1$ sufficiently close to $1$, and shrinking $\eps>0$ if necessary, this holds for the point
\[
  z := (\pi/2-d_0/2-3\eps, q) \in M.
\]

For small $\lambda\geq 0$, one can define the SdS metric $g_{p_i,\lambda\bhm_i}$, extended across the cosmological horizon and defined in $(0,\pi/2]_s$ times a $2 r_0$-neighborhood of $p_i$ inside $\Sph^3$, in such a manner that as $\lambda\to 0$, the weighted difference $\tau^{-3}(g_\dS-g_{p_i,\lambda\bhm_i})$ (cf.\ equation~\eqref{EqSODiff}) converges smoothly to $0$ as a section of $S^2\,{}^0 T^*M$ away from the line $L_i:=\{(s,p_i)\colon s\in(-\pi/2,\pi/2]\}$.\footnote{%
A systematic and more precise way of accomplishing this is to use geometric microlocal techniques \cite{MelroseDiffOnMwc}. For a single SdS black hole centered at $p\in\Sph^3$, one starts with the total space $[0,\lambda_0)\times M$ and blows up $[0,\lambda_0)\times\{p\}$ and then $\{0\}\times L$, $L=(0,\pi/2]\times\{p\}$. The first blow-up resolves the singular nature---due to its $r$-dependence---of the SdS metric near $p$, and the second blow-up resolves the event horizon, whose $r$-coordinate goes to $0$ roughly linearly with $\lambda$. The family of SdS metrics with mass $\lambda\bhm$ can then be defined as a smooth section of the pullback of $S^2\,{}^0 T^*M$ to this resolved space, and, crucially, in such a manner that it equals the de~Sitter metric on the lift of $\lambda=0$.%
}
We claim that for sufficiently small $\lambda>0$, we can do the SdS gluing with parameters $\sD_\lambda$ in such a way that
\begin{enumerate}
\item\label{ItSL1} the point $z$ is contained in the maximal globally hyperbolic development of the glued metric $g_\lambda$ with respect to $\Sigma_0$ (and $\Sigma_0$ is spacelike for $g_\lambda$),
\item $\bar z_1,\bar z_2\in I^+(z)$ with respect to $g_\lambda$,
\item near $\bar z_j$, $g_\lambda$ is equal to $g_{p_j,\bhm_j}$ for $j=1,2$, and
\item\label{ItSL4} $\bar z_j$ lies \emph{inside} the cosmological horizon of the SdS black hole associated with the point $p_j$.
\end{enumerate}

To begin, let $\lambda_0>0$ be a small fixed constant. Consider then the naively glued metric $g_{\lambda,0}=\chi_0 g_\dS+\sum_{i=1}^N\chi_i g_{p_i,\lambda\bhm_i}$ analogously to~\eqref{EqSOGlued}; this fails to solve the Einstein vacuum equation by the amount $P_0(g_{\lambda,0})=\lambda\tau^4\Err(\lambda)$, where $\Err(\lambda)\in\CI([0,\lambda_0)\times M;S^2\,{}^0 T^*M)$ depends smoothly on $\lambda\geq 0$, and whose leading order term at $\pa M$ can be computed using Lemma~\ref{LemmaSO1} and equation~\eqref{EqSOConf}; in particular, the leading order ($\tau^0$) coefficient of $\Err$ is $\lambda$-independent. Here and below, we take $\lambda_0>0$ to be a small fixed constant. Thus, we can take the solution of the divergence equation $-(\Lambda/3)3\delta_{g_{\Sph^3}}k_\lambda=-\lambda\Err(\lambda)_{N T}$ (cf.\ equation~\eqref{EqSODivEq}), to be $k=\lambda k_1$ for some \emph{fixed} $k_1\in\CI(\Sph^3;S^2 T^*\Sph^3)$. For small $\lambda$, we then work with the background metric $g_\lambda^0=g_{\lambda,0}+k_\lambda$.

The remainder of the formal part of the gluing construction does not depend on any further choices; the Borel summation in the proof of Proposition~\ref{PropSGFormal} can be defined to produce a metric correction $\tilde g_{\lambda,1}\in\lambda\tau^5\CI([0,\lambda_0)\times M;S^2\,{}^0 T^*M)$ with support in a small fixed neighborhood of $\pa M$. The resulting formal solution $g_{\lambda,1}=g_\lambda^0+\tilde g_{\lambda,1}$ (cf.\ the statement of Proposition~\ref{PropSCNonlinear}) satisfies the gauge-fixed Einstein equation with error $P(g_{\lambda,1};g^0_\lambda)\in\lambda\tau^\infty\CI([0,\lambda_0)\times M;S^2\,{}^0 T^*M)$; the support of this error is disjoint from the $\cN_i$. But then, since $g_{\lambda,1}$ and $g_\lambda^0$ converge smoothly to the de~Sitter metric in the backwards domain of dependence of $\Sigma_0$, Cauchy stability for the solution $\tilde g_{\lambda,2}\in\lambda\tau^\infty\CI$ of the quasilinear equation $P(g_{\lambda,1}+\tilde g_{\lambda,2};g^0_\lambda)=0$ implies that $g_\lambda=g_{\lambda,1}+\tilde g_{\lambda,2}$ indeed exists (uniquely, by domain of dependence considerations) on a sufficiently large subset of $M$ so that the requirements~\eqref{ItSL1}--\eqref{ItSL4} are indeed met.

\subsection{Necessity of the balance condition}
\label{SsSB}

In this section, we prove Theorem~\ref{ThmSUniq}. Thus, suppose that $\Ric(g)-\Lambda g=0$ for a metric $g$ of the form
\[
  g = g_\dS + \sum_{i=1}^N \chi_i(g_{p_i,\bhm_i}-g_\dS) + \tau^3(\log\tau)g_\ell + \tau^3 g_3 + \tilde g,\quad \tilde g\in\cA^{3+\eps}(M;S^2\,{}^0 T^*M),
\]
where $\chi_i\in\CI(\pa M)$ is a cutoff localizing to a small neighborhood of $p_i$; here $g_\ell,g_3\in\CI(\pa M;S^2\,{}^0 T^*_{\pa M}M)$, and $g_\ell,g_3,\tilde g$ have supports disjoint from the $p_i$. Then, with $L_0=D_{g_\dS}P_0$ as in~\eqref{EqSLinOp}, Lemma~\ref{Lemma0EEin} gives
\begin{equation}
\label{EqSBf}
  f := L_0\bigl(\tau^3(\log\tau)g_\ell + \tau^3 g_3+\tilde g\bigr) + \tau^4\Err \in \cA^{6-\delta}\quad\forall\,\delta>0,
\end{equation}
where the normal-tangential component $\Err_{N T}$ of $\Err\in\CI(\pa M;S^2\,{}^0 T^*M)$ takes the form given in Lemma~\ref{LemmaSO2}.

Note that, for any $g_\ell,g_3,\tilde g$ in the above function spaces, we have $f\in\tau^3(\log\tau)\CI+\tau^3\CI+\cA^{3+\eps}$. Using~\eqref{EqSBf}, its $\tau^3\log\tau$ coefficient is $I(L_0,3)g_\ell=0$. In view of~\eqref{EqSInd3} (with $h_s,\tau_s$ replaced by $h,\tau$, where $h$ is the boundary metric~\eqref{EqSBdyMet}), we thus have, in the bundle splitting~\eqref{Eq0ESplitTr},
\[
  g_\ell = (3 u,\eta,u,k),
\]
where $u\in\CI(\pa M)$, $\eta\in\CI(\pa M;T^*\pa M)$, and $k\in\CI(\pa M;S^2 T^*\pa M)$ with $\tr_h k=0$ are supported away from the $p_i$.

Using~\eqref{Eq0BIndLog} and Corollary~\ref{Cor0EInd}, we then compute the $\tau^3$ coefficient of $f$ as
\[
  0 = I(L_0,3)g_3 + \pa_\lambda I(L_0,\lambda)|_{\lambda=3} g_\ell = I(L_0,3)g_3 + (\Lambda/3) \bigl( -3 u, 0, -3 u, 3 k \bigr).
\]
Since the pure trace part of the tangential-tangential component of $I(L_0,3)$ always vanishes, we must have $k=0$. But then we can then write
\[
  g_3=g_{3 0}+g_{3 1}, \quad I(L_0,3)g_{3 0}=0,\quad g_{3 1}=(u,0,0,0),
\]
with $g_{3 1}$ defined so that it solves $I(L_0,3)g_{3 1}+\pa_\lambda I(L_0,\lambda)|_{\lambda=3}g_\ell=0$.

Lastly then, projecting to the normal-tangential component, we have $L_0(\tilde g)_{N T}\in\cA^{4+\eps}$ by Corollary~\ref{Cor0EInd}, hence the $\tau^4$ component of $f$ is equal to
\begin{align*}
  0&=\Err_{N T} + \pa_\lambda I(L_0[\tau],\lambda)|_{\lambda=3}g_\ell + I(L_0[\tau],3)g_{3 1} + I(L_0[\tau],3)g_{3 0} \\
  & = \Err_{N T} + (\Lambda/3)(-2 d_X)u + (\Lambda/3)2 d_X u + (\Lambda/3)(-3\delta_h(g_{3 0})_{T T 0}) \\
  & = \Err_{N T} - \Lambda\delta_h(g_{3 0})_{T T 0},
\end{align*}
where $(g_{3 0})_{T T 0}\in\CI(\pa M;S^2 T^*\pa M)$ is the trace-free part of the tangential-tangential component of $g_{3 0}$. Integrating this against a conformal Killing vector field $V$, an integration by parts implies that the second term does not contribute, and therefore $\int_{\pa M} V(\Err_{N T})\,|d h|=0$ for all $V\in\mathfrak{conf}_3$. An application of Lemma~\ref{LemmaSO2} concludes the proof of Theorem~\ref{ThmSUniq}.

\subsection{Gluing with noncompact spatial topology}
\label{SsSNc}

The balance condition in Definition~\ref{DefSBalance} captures the orthogonality of the leading order error term to conformal Killing vector fields on $\Sph^3$. If, however, we allow the trace-free 2-tensor $k$ in~\eqref{EqSODivEq} to be blowing up sufficiently fast at a point $p_\infty\in\Sph^3$ distinct from the $p_i$, this obstruction disappears, since elements of the relevant cokernel now need to vanish at sufficiently high order at $p_\infty$; since conformal Killing vector fields on the sphere vanish at most quadratically at any given point, this cokernel is empty.

A more natural way to phrase this is to pass to the upper half space picture of de~Sitter space, $M_u=[0,\infty)_{\tilde\tau}\times\R^3_{\tilde x}$, see~\eqref{Eq0BdSNoncpt2}; the point $p_\infty$ is the point at infinity within the conformal boundary $\pa M_u$, and we need to solve equation~\eqref{EqSODivEq}, with $h=(3/\Lambda)g_{\R^3}$ now a constant multiple of the Euclidean metric, and the error term having compact support in $\tilde x$. We can \emph{always} solve this, with the solution $k$ having support disjoint from the $\ol{V_{p_i}}$, if we allow $k$ to be nonzero in $|\tilde x|\gg 1$ and allow for $\la\tilde x\ra^{-2+\eps}$ decay, $\eps>0$. More precisely, using the function spaces of \cite[Appendix~A]{ChruscielDelayMapping}, we can find $k\in H_{\phi,\psi}^1(g_{\R^3})$ with $\phi=\la\tilde x\ra$ and $\psi=\la\tilde x\ra^{1/2-2\eps}$. (Refining the weights to be exponential at the boundary of the domain in which one wants $k$ to be supported enforces the correct support of $k$.) Indeed, since $\delta_{g_{\R^3}}\colon H_{\phi,\psi}^1(g_{\R^3})\to L^2_{\psi\phi}(g_{\R^3})$ by~\cite[(A.4)]{ChruscielDelayMapping}, the relevant cokernel (for $s=1$) consists of conformal Killing 1-forms $\omega$ on $(\R^3,g_{\R^3})$ lying in $L^2_{\psi^{-1}\phi^{-1}}(g_{\R^3})$. But since all such $\omega$ have size at least $1$ near infinity, and since $\int |1|^2 \psi^{-2}\phi^{-2}|d g_{\R^3}|\sim \int \la\tilde x\ra^{-3+4\eps} r^2\,d r$ diverges, the space of such $\omega$ is trivial.

The remainder of the gluing construction as before; however, in the final step, the domain of existence of the quasilinear equation might shrink to zero as $|\tilde x|\to\infty$. (Even when not gluing \emph{any} black hole into the upper half space model  $M_u$, we point out that past directed null-geodesics leave $M_u$ in finite affine time; see \cite[Figure~7]{HintzZworskiHypObs}.) In summary:

\begin{thm}
\label{ThmSNc}
  Let $N\in\N$, and let $p_1,\ldots,p_N\in\R^3$, $\bhm_1,\ldots,\bhm_N\in\R$. Suppose $V_{p_i}\subset\R^3$ is a punctured neighborhood of $p_i$, and suppose $\ol{V_{p_i}}\cap\ol{V_{p_j}}=\emptyset$ for $i\neq j$. Then there exists a neighborhood $U$ of $\pa M_u\setminus\{p_1,\ldots,p_N\}$ and a Lorentzian 0-metric $g\in\CI(U;S^2\,{}^0 T^*_U M_u)$ satisfying the properties~\eqref{ItSGlueEin}--\eqref{ItSGluePert} of Theorem~\usref{ThmSGlue}.
\end{thm}

\begin{rmk}
\label{RmkSNcGeneral}
  More generally, we can glue any spacetime into de~Sitter space whose metric is defined in an interval (in $\tilde\tau$) times an annulus (in $\tilde x$) around a point $p\in\pa M_u$, provided the metric is asymptotic to $g_\dS$ at a rate $\tau^3$ in this region. This can be further relaxed, but we will not pursue this further.
\end{rmk}

\section{Multi-Kerr--de~Sitter spacetimes}
\label{SK}

The goal is to glue several Kerr--de~Sitter (KdS) black holes into a neighborhood of the future conformal boundary of global de~Sitter space; we thus continue to work on the manifold $M=[0,1)_\tau\times\Sph^3$ as in~\eqref{EqSMfd}.

\subsection{Kerr--de~Sitter metrics in corotating coordinates; parameterization}
\label{SsKP}

We recall the KdS metric with parameters $\bhm\in\R$ and $\bha\in\R$ in the form
\begin{equation}
\label{EqKMetric}
  g_{\bhm,\bha} = -\frac{\Delta_r}{\rho^2}\Bigl(d t_0-\frac{\bha\sin^2\theta_0}{\Delta_0}d\phi_0\Bigr)^2 + \frac{\rho^2}{\Delta_r}d r_0^2 + \frac{\rho^2}{\Delta_\theta}d\theta_0^2 + \sin^2\theta_0 \frac{\Delta_\theta}{\rho^2}\Bigl(\bha\,d t_0-\frac{r_0^2+\bha^2}{\Delta_0}d\phi_0\Bigr)^2,
\end{equation}
where we define (omitting the dependence on $\bhm,\bha$ from the notation)
\begin{alignat*}{3}
  \Delta_r &= (r_0^2+\bha^2)\Bigl(1-\frac{\Lambda r_0^2}{3}\Bigr)-2\bhm r_0,\qquad &
  \Delta_\theta &= 1+\frac{\Lambda}{3}\bha^2\cos^2\theta_0, \\
  \rho^2 &= r_0^2+\bha^2\cos^2\theta_0,\qquad &
  \Delta_0 &= 1+\frac{\Lambda}{3}\bha^2.
\end{alignat*}
(This matches the expression in~\cite[Equations~(5.2)--(5.4)]{SchlueCosmological} upon adding the subscripts `$0$', and differs from that in~\cite[Equation~(3.12)]{HintzVasyKdSStability} only by a constant rescaling of $t$ by $\Delta_0$.) Following~\cite[Appendix~B]{SchlueCosmological},\footnote{For consistency with~\S\ref{SS}, the roles of $t_0,r_0,\dots$ and $t,r,\dots$ are reversed compared to the reference.} we recall the coordinate change which displays $g_{\bhm,\bha}$ as a perturbation of the de~Sitter metric~\eqref{Eq0BdSStatic2} up to terms of size $r^{-3}$ (as uniformly degenerate symmetric 2-tensors). Thus, under the change of coordinates\footnote{The definition of $r^2$ implies that $r_0^2\cos^2\theta_0\leq r^2$, hence $\theta$ is well-defined.}
\begin{alignat*}{3}
  t &= t_0, &\qquad
  \phi &= \phi_0-\frac{\Lambda}{3}\bha t_0, \\
  r^2 &= \frac{1}{\Delta_0}\bigl(r_0^2\Delta_\theta+\bha^2\sin^2\theta_0\bigr), &\qquad
  r\cos\theta &= r_0\cos\theta_0.
\end{alignat*}
the de~Sitter metric $g_\dS$ in~\eqref{Eq0BdSStatic2} takes the form
\begin{align*}
  g_\dS &= \Bigl(\frac{\Lambda}{3}(r_0^2+\bha^2\sin^2\theta_0)-1\Bigr)d t_0^2 + \frac{\rho^2}{(r_0^2+\bha^2)\bigl(1-\tfrac{\Lambda r_0^2}{3}\bigr)}d r_0^2 + \frac{\rho^2}{\Delta_\theta^2}d\theta_0^2 \\
    &\quad\qquad - \frac{2\Lambda}{3}\bha\frac{r_0^2+\bha^2}{\Delta_0}\sin^2\theta_0\,d t_0\,d\phi_0 + \frac{r_0^2+\bha^2}{\Delta_0}\sin^2\theta_0\,d\phi_0^2.
\end{align*}
Therefore
\begin{equation}
\label{EqKDiff}
  g_{\bhm,\bha} = g_\dS + c_{\bhm,\bha}, \quad c_{\bhm,\bha} = \frac{2\bhm r_0}{\rho^2}\Bigl(d t_0-\frac{\bha\sin^2\theta_0}{\Delta_0}d\phi_0\Bigr)^2 + \frac{2\bhm r_0\rho^2}{\Delta_r|_{\bhm=0}\Delta_r}d r_0^2.
\end{equation}
We then compute:
\begin{lemma}
\label{LemmaKDiff}
  Let $\tau_s=r^{-1}$, and denote by $\CI$ the space of functions which are smooth in $(t,\tau_s,\theta,\phi)$, $\tau_s=r^{-1}$, down to $\tau_s=0$. The symmetric 2-tensor $c_{\bhm,\bha}$ in~\eqref{EqKDiff} then has components, modulo $\tau_s^4\CI$, given by
  \begin{alignat*}{3}
    r^2(c_{\bhm,\bha})_{r r} &\equiv r^{-3}\Bigl[\Bigl(\frac{\Delta_\theta}{\Delta_0}\Bigr)^{3/2}\frac{18\bhm}{\Lambda^2}\Bigr], &\qquad
    r^{-2}(c_{\bhm,\bha})_{t t} &\equiv r^{-3}\Bigl[\sqrt{\frac{\Delta_\theta}{\Delta_0}}2\bhm\frac{\Delta_\theta^2}{\Delta_0^2}\Bigr], \\
    r^{-2}(c_{\bhm,\bha})_{t\phi} &\equiv r^{-3}\Bigl[\sqrt{\frac{\Delta_\theta}{\Delta_0}}\cdot(-2\bhm)\bha\sin^2\theta_0\frac{\Delta_\theta}{\Delta_0^2}\Bigr], &\qquad
    r^{-2}(c_{\bhm,\bha})_{\phi\phi} &\equiv r^{-3}\Bigl[\sqrt{\frac{\Delta_\theta}{\Delta_0}}2\bhm\frac{\bha^2\sin^4\theta_0}{\Delta_0^2}\Bigr].
  \end{alignat*}
  Furthermore, $(c_{\bhm,\bha})_{r t}=(c_{\bhm,\bha})_{r\phi}=(c_{\bhm,\bha})_{t\theta}=(c_{\bhm,\bha})_{\theta\phi}=0$, and
  \[
    r^{-2}(c_{\bhm,\bha})_{\theta\theta} \in \tau_s^5\CI, \quad
    (c_{\bhm,\bha})_{r\theta} \in \tau_s^3\CI.
  \]
\end{lemma}
\begin{proof}
  Since $1\leq\Delta_0,\Delta_\theta\leq 1+\Lambda\bha^2/3$, we record that
  \begin{equation}
  \label{EqKDiffr0r}
    \frac{r}{r_0} \equiv \sqrt{\frac{\Delta_\theta}{\Delta_0}} \bmod \tau_s^2\CI,
  \end{equation}
  and in particular~$r/r_0,r_0/r\in\CI$. Now, direct calculations give
  \begin{alignat*}{3}
    \pa_{t_0} &= \pa_t-\frac{\Lambda}{3}\bha\pa_\phi, &\quad
    \pa_{r_0} &= \frac{\Delta_\theta r_0}{\Delta_0 r}\pa_r - \frac{\bha^2\sin^2\theta_0\cos\theta_0}{\Delta_0 r^3\sin\theta}\pa_\theta, \\
    \pa_{\phi_0} &= \pa_\phi, &\quad
    \pa_{\theta_0} &= \frac{\bha^2}{\Delta_0 r}\Bigl(1-\frac{\Lambda r_0^2}{3}\Bigr)\cos\theta_0\sin\theta_0\pa_r \\
    &&&\quad\qquad + \frac{r_0\sin\theta_0}{r\sin\theta}\Bigl(1+\frac{\bha^2\cos^2\theta_0}{\Delta_0 r^2}\Bigl(1-\frac{\Lambda r_0^2}{3}\Bigr)\Bigr)\pa_\theta.
  \end{alignat*}
  The main structure of the right column is captured by
  \[
    \begin{pmatrix} \pa_{r_0} \\ \pa_{\theta_0} \end{pmatrix} = \begin{pmatrix} \frac{\Delta_\theta r_0}{\Delta_0 r} & \tau_s^3\CI \\ \tau_s^{-1}\CI & \CI \end{pmatrix} \begin{pmatrix} \pa_r \\ \pa_\theta \end{pmatrix},
  \]
  with the bottom right entry invertible. Therefore,
  \[
    \pa_t=\pa_{t_0}+\frac{\Lambda}{3}\bha\pa_{\phi_0},\quad \pa_\phi=\pa_{\phi_0},\quad
    \begin{pmatrix} \pa_r \\ \pa_\theta \end{pmatrix} = \begin{pmatrix} \frac{\Delta_0 r}{\Delta_\theta r_0} & \tau_s^3\CI \\ \tau_s^{-1}\CI & \CI \end{pmatrix} \begin{pmatrix} \pa_{r_0} \\ \pa_{\theta_0} \end{pmatrix}.
  \]
  Note also that $\rho^2\equiv r_0^2\mod\CI$, hence $2\bhm r_0/\rho^2\equiv 2\bhm/r_0\bmod\tau_s\CI$, and moreover $\Delta_r\equiv-\Lambda r_0^4/3\bmod\tau_s^{-2}\CI$; therefore,
  \[
    c_{\bhm,\bha} \in \Bigl(\frac{2\bhm}{r_0}+\tau_s^3\CI\Bigr)\Bigl(d t_0-\frac{\bha\sin^2\theta_0}{\Delta_0}d\phi_0\Bigr)^2 + \Bigl(\frac{18\bhm}{\Lambda^2 r_0^5}+\tau_s^7\CI\Bigr)d r_0^2.
  \]
  Thus, for instance, we have $r^2(c_{\bhm,\bha})_{r r} \equiv r^2\frac{\Delta_0^2 r^2}{\Delta_\theta^2 r_0^2}\frac{18\bhm}{\Lambda^2 r_0^5}$, which gives the stated result upon using~\eqref{EqKDiffr0r}. The other components are calculated similarly.
\end{proof}

Note that $r\pa_r=-\tau_s\pa_{\tau_s}$ and $r^{-1}\pa_\bullet=\tau_s\pa_\bullet$ for $\bullet=t,\theta,\phi$. Let now $r_+$ be such that $\inf_{\theta\in(0,\pi)}r_0(r_+,\theta)$ is larger than the largest positive real root of $\Delta_r$ (as a function of $r_0$) if one exists, and otherwise fix any $r_+>0$. Define the manifold
\[
  M_{\bhm,\bha,s} := [0,r_+^{-1})_{\tau_s} \times \R_t \times \Sph^2_{\theta\phi}
\]
where $\tau_s=r^{-1}$; then the lemma implies that
\[
  c_{\bhm,\bha} \in \tau_s^3 \CI(M_{\bhm,\bha,s}\cap M_s,S^2\,{}^0 T^*M_s)
\]
on the common domain of definition of the KdS metric and the de~Sitter metric, cf.\ \eqref{Eq0BdSStatic3Mfd}. (We leave it to the reader to check that $c_{\bhm,\bha}\in\tau_s^3\CI$ also at the poles of $\Sph^2$ where the polar coordinates break down.) In view of~\eqref{Eq0BdSStatic3}, we in particular have $g_{\bhm,\bha}\in\CI$ on the common domain of definition.

At $\tau_s=0$, the limits $t\to\infty$ and $t\to-\infty$ correspond to $\tilde R\to 0$ and $\tilde R\to\infty$, respectively, in the upper half space coordinates~\eqref{Eq0BdSNoncpt}. Therefore, as in the SdS case, $g_{\bhm,\bha}$ is defined in a neighborhood of $U_{p_0}:=\Sph^3\setminus\{p_0,-p_0\}$ where $p_0=(1,0,0,0)\in\Sph^3\subset\R^4$ is the point defined by $\tilde R=0$ inside $\tilde\tau=0$; it describes a KdS black hole rotating in $p_0^\perp$ around an axis, which we fix to be
\[
  \hat\bha_0=(0,0,0,1),
\]
with specific angular momentum $\bha$.

We wish to define KdS metrics located at other points on the future conformal boundary $\pa M$. To this end, we use a parameterization of the KdS family by triples
\begin{equation}
\label{EqKParam}
  (p,\bhm,\fa),\qquad p\in\Sph^3,\ \bhm\in\R,\ \fa\in\so_4,\ \fa p=0;
\end{equation}
here, we identify $\so_4$ both with $\{A\in\R^{4\times 4}\colon A+A^T=0\}$ and the space $\mathfrak{iso}_3$ of Killing vector fields on $\Sph^3$ where $A\in\R^{4\times 4}$ corresponds to the vector field $\frac{d}{d s}e^{s A}|_{s=0}$ on the unit sphere $\Sph^3\subset\R^4$. Thus, viewing $\fa\in\so_4$ as a matrix, the condition $\fa p=0$ means that $p\in\Sph^3\subset\R^4$ lies in its kernel; viewing $\fa\in\mathfrak{iso}_3$, it means that the vector field $\fa$ vanishes at $p$.

\begin{definition}
\label{DefKParam}
  We call a triple $(p,\bhm,\fa)\in\Sph^3\times\R\times\so_4$ \emph{admissible} if $\fa p=0$.
\end{definition}

For $\fa,\fa_1,\fa_2\in\so_4$, viewed as $4\times 4$ matrices, we denote
\begin{equation}
\label{EqKInnerProd}
  \la\fa_1,\fa_2\ra := \half\sum_{i,j=1}^4 (\fa_1)_{i j}(\fa_2)_{i j} = \sum_{i<j} (\fa_1)_{i j}(\fa_2)_{i j},\qquad
  |\fa|^2 := \la\fa,\fa\ra.
\end{equation}
(Invariantly, $\la\cdot,\cdot\ra$ is $(-\half)$ times the Killing form on $\so_4$.) Given an admissible triple $(p,\bhm,\fa)$, we define the metric
\begin{equation}
\label{EqKParamMetric}
  g_{p,\bhm,\fa}
\end{equation}
as a smooth Lorentzian 0-metric near $U_p=\pa M\setminus\{p,-p\}$ as follows. First, if $\fa=0$, we let $g_{p,\bhm,0}=g_{p,\bhm}$ be equal to the SdS metric with mass $\bhm$ centered at $p$, as defined in the paragraph leading up to~\eqref{EqSParamMetric}. Otherwise, $\fa\in\R^{4\times 4}$ induces a nontrivial skew-adjoint linear transformation on $p^\perp\subset\R^4$, equivalently a rotation vector field, around an axis $\hat\bha\in p^\perp$, $|\hat\bha|=1$, with amplitude $\bha:=|\fa|$. Choose then an element $R_1\in SO(4)$ with $R_1 p=p_0$, and then (noting that $R_1\hat\bha\in p_0^\perp$) an element $R_2\in SO(4)$ with $R_2 p_0=p_0$ so that $R_2(R_1\hat\bha)=\hat\bha_0$; put $R=R_2 R_1$. Note that the properties $R p=p_0$ and $R\hat\bha=\hat\bha_0$ determine $R$ uniquely up to multiplication from the left by a rotation fixing $p_0$ and $\hat\bha_0$, thus a rotation $\phi\mapsto\phi+\phi'$ for some $\phi'\in\R$---which is an isometry of $g_{\bhm,\bha}$. We then define $g_{p,\bhm,\bha}$ as the pullback of $g_{\bhm,\bha}$ along the map $M\to M$, $(\tau,\psi)\mapsto(\tau,R(\psi))$. In particular, this parameterizes $g_{\bhm,\bha}$ as
\begin{equation}
\label{EqKParamMetric0}
  g_{\bhm,\bha} = g_{p_0,\bhm,\fa_0},\quad \fa_0 = \begin{pmatrix} 0 & 0 & 0 & 0 \\ 0 & 0 & \bha & 0 \\ 0 & -\bha & 0 & 0 \\ 0 & 0 & 0 & 0 \end{pmatrix}.
\end{equation}

\subsection{Gluing theorem}

With the KdS metrics $g_{\bhm,p,\fa}$ defined as in~\S\ref{SsKP}, we are ready to state the gluing theorem, which holds subject to a balance condition generalizing Definition~\ref{DefSBalance}; it involves the \emph{effective mass} of an admissible triple $b=(p,\bhm,\fa)$, defined as
\[
  \bhm_{\rm eff}(b) := \frac{\bhm}{(1+\Lambda|\fa|^2/3)^2}.
\]

\begin{definition}
\label{DefKBalance}
  Let $N\in\N$. We say that a collection $\{b_1,\ldots,b_n\}$ of admissible triples $b_i=(p_i,\bhm_i,\fa_i)$ is \emph{balanced} if the $p_i$ are pairwise distinct and if, regarding $\Sph^3$ as the unit sphere inside $\R^4$, the following relations hold:
  \begin{subequations}
  \begin{align}
    \sum_{i=1}^N \bhm_{\rm eff}(b_i) p_i &= 0 \in \R^4, \\
    \sum_{i=1}^N \bhm_{\rm eff}(b_i)\fa_i &= 0 \in \so_4\subset\R^{4\times 4}.
  \end{align}
  \end{subequations}
\end{definition}

\begin{thm}
\label{ThmKGlue}
  Let $N\in\N$, and suppose $\{b_1,\ldots,b_N\}\subset\Sph^3\times\R\times\so_4$ is balanced, $b_i=(p_i,\bhm_i,\fa_i)$. Suppose $V_{p_i}\subset U_{p_i}$ is a ball around $p_i$ with the point $p_i$ removed, and suppose $\ol{V_{p_i}}\cap\ol{V_{p_j}}=\emptyset$ for $i\neq j$. Then there exist a neighborhood $U$ of $\pa M\setminus\{p_1,\ldots,p_N\}$ and a Lorentzian 0-metric $g\in\CI(U;S^2\,{}^0 T_U^*M)$ with the following properties:
  \begin{enumerate}
  \item\label{ItKGlueEin} $g$ satisfies the Einstein vacuum equation $\Ric(g)-\Lambda g=0$;
  \item\label{ItKGlueKdS} near $V_{p_i}$, we have $g=g_{p_i,\bhm_i,\fa_i}$;
  \item\label{ItKGluePert} $g$ is $\cO(\tau^3)$-close to the de~Sitter metric: $g-g_\dS\in\tau^3\CI(U;S^2\,{}^0 T_U^*M)$.
  \end{enumerate}
\end{thm}

In the special case that $b_i=(p_i,\bhm_i,0)$ for all $i$, this reduces to Theorem~\ref{ThmSGlue}.

\begin{rmk}
\label{RmkKBH}
  A remark analogous to Remark~\ref{RmkSBH} applies also in the Kerr--de~Sitter setting: if the black hole parameters are subextremal, one can extend the glued Kerr--de~Sitter metrics across their cosmological and event horizons. See e.g.\ \cite[\S3.2]{HintzVasyKdSStability}. For small masses, the domain of existence of $g$ can be shown to include the interaction of several black holes by following the arguments at the end of~\S\ref{SsSC}.
\end{rmk}

\begin{rmk}
  If one passes to the upper half space model $M_u$, there are no obstructions to gluing anymore, analogously to Theorem~\ref{ThmSNc}.
\end{rmk}

The main part of the proof of Theorem~\ref{ThmKGlue} is the calculation of the obstruction for solving the divergence equation~\eqref{EqSODivEq}. First, we compute the failure of the Einstein equation for a naive gluing of a single KdS black hole. Let $P_0=2(\Ric-\Lambda)$ and $L_0=D_{g_\dS}P_0$ as in~\eqref{EqSLinOp}.

\begin{lemma}
\label{LemmaKO1}
  Let $\chi\in\CI(\R_t)$ be identically $1$ for large $t$, and put $g_0=\chi(t)g_{\bhm,\bha}+(1-\chi(t))g_\dS$. With $P_0=2(\Ric-\Lambda)$ as in~\eqref{Eq0EEinOp}, we then have $P_0(g_0)=\tau_s^4\Err_{s,0}\bmod\tau_s^5\CI$, where $\Err_{s,0}=2\frac{d\tau_s}{\tau_s} \otimes_s \frac{(\Err_{s,0})_{N T}}{\tau_s}+\Err'_{s,0}$ with
  \[
    (\Err_{s,0})_{N T} = \frac{18\bhm}{\Lambda\Delta_0^2}\Delta_\theta\sqrt{\frac{\Delta_\theta}{\Delta_0}}\chi'(t)\Bigl(\bigl(\Delta_\theta-\tfrac13\Delta_0\bigr)d t-\bha\sin^2\theta_0\,d\phi\Bigr),
  \]
  and $\Err'_{s,0}=\tau_s^4 I(L_0,4)\tilde c$ for some $\tilde c\in\CI(\pa M_s;S^2\,{}^0 T_{\pa M_s}^* M_s)$ with $\supp\tilde c\subset\supp d\chi$.
\end{lemma}
\begin{proof}
  Recall from~\eqref{Eq0BdSStaticBdy} the metric $h_s=(\Lambda^2/9)d t^2+(\Lambda/3)g_{\Sph^2}$ induced on the boundary $\pa M_s$ by $g_\dS$ and the boundary defining function $r^{-1}$. In the splitting~\eqref{Eq0ESplit} with $h,\tau$ replaced by $h_s,\tau_s$, the leading order components of $c_{\bhm,\bha}=g_{\bhm,\bha}-g_\dS$ are then, by Lemma~\ref{LemmaKDiff},
  \begin{align*}
    (\gamma_3)_{N N} &= \bigl(r^3\cdot r^2(c_{\bhm,\bha})_{r r}\bigr)|_{\tau_s=0} = \Bigl(\frac{\Delta_\theta}{\Delta_0}\Bigr)^{3/2}\frac{18\bhm}{\Lambda^2}, \\
    (\gamma_3)_{T T} &= r^3\cdot\Bigl(r^2 \bigl((c_{\bhm,\bha})_{t t}\,d t^2+2(c_{\bhm,\bha})_{t\phi}\,d t\,d\phi+(c_{\bhm,\bha})_{\phi\phi}\,d\phi^2\bigr)\Bigr)\Big|_{\tau_s=0},
  \end{align*}
  and $(\gamma_3)_{N T}$ is a smooth 1-form on $\R_t\times\Sph^2$ whose precise form we do not need.

  Since $0=P_0(g_{\bhm,\bha})\equiv L_0(c_{\bhm,\bha}) \bmod \tau_s^6\CI$ as in the proof of Lemma~\ref{LemmaSO1}, and since $c_{\bhm,\bha}\equiv r^{-3}\gamma_3\bmod\tau_s^4\CI$, we conclude that $I(L_0,3)\gamma_3=0$. In view of~\eqref{EqSInd3}, this implies the relationship $(\gamma_3)_{N N}=\tr_{h_s}(\gamma_3)_{T T}$ (using that $3=\tr_{h_s}h_s$).\footnote{This can also be checked directly. Indeed, the equality of $\tr_{h_s}(\gamma_3)_{T T}=\frac{9}{\Lambda^2}(\gamma_3)_{t t}+\frac{3}{\Lambda}\sin^{-2}\theta\,(\gamma_3)_{\phi\phi}$ and $(\gamma_3)_{N N}$ is equivalent to $\Delta_\theta^2+\sin^{-2}\theta\,\frac{\Lambda}{3}\bha^2\sin^4\theta_0=\Delta_\theta\Delta_0$ and thus to $\Delta_\theta=\frac{\sin^2\theta_0}{\sin^2\theta}$; this is easily verified by plugging in $\sin^2\theta=1-\frac{r_0^2}{r^2}\cos^2\theta_0=1-\frac{\Delta_0}{\Delta_\theta}\cos^2\theta_0$, which holds at $\tau_s=0$.} Therefore, the trace-free part $(\gamma_3)_{T T 0}$ in the refined splitting~\eqref{Eq0ESplitTr} is given by
  \[
    (\gamma_3)_{T T 0} = (\gamma_3)_{T T} - \frac{1}{3}(\gamma_3)_{N N}h_s.
  \]
  By following the calculation~\eqref{EqSO1Calc}, the normal-tangential component of $\Err_{s,0}$ is thus
  \[
    (\Err_{s,0})_{N T} = -\Lambda\cdot 9\Lambda^{-2}\chi'(t)\cdot\bigl(-\iota_{\pa_t}(\gamma_3)_{T T 0}\bigr),
  \]
  which we can compute by means of Lemma~\ref{LemmaKDiff}.

  Regarding the remaining components of $\Err_{s,0}$, we note that they lie in the range of the third column of the operator~\eqref{EqSIndlot3}. But by Corollary~\ref{Cor0EInd}, we have
  \[
    3\Lambda^{-1}I(L_0,4) = \begin{pmatrix} 6 & 0 & -24 & 0 \\ 0 & 0 & 0 & 0 \\ 2 & 0 & -8 & 0 \\ 0 & 0 & 0 & 4 \end{pmatrix},
  \]
  whose range is thus spanned by $(1,0,\tfrac13,0)$ and $\ker\tr_h$, and hence contains the range of the third column of~\eqref{EqSIndlot3}.
\end{proof}

Since the components of $\Err_{s,0}$ other than the normal-tangential component can thus be solved away pointwise on $\pa M_s$ (modulo one order down, i.e.\ modulo $\tau_s^5\CI$) with a $\tau_s^4\CI$ correction, the only obstruction for gluing is again the integral of $(\Err_{s,0})_{N T}$ against conformal Killing vector fields on $(\pa M_s,h_s)=(\R_t\times\Sph^2,h_s)$ as in~\eqref{EqSOIndSolvCrit}. The volume density in these integrals is
\[
  |d h_s| = \frac{\Lambda^2}{9}d t\,\sin\theta\,d\theta\,d\phi = \frac{\Lambda^2}{9}\sqrt{\frac{\Delta_0}{\Delta_\theta}}\Delta_\theta^{-1}\,d t_0\,\sin\theta_0\,d\theta_0\,d\phi_0,
\]
since at $\tau_s=0$ we have, using~\eqref{EqKDiffr0r},
\begin{align*}
  \sin\theta\,d\theta&=-d(\cos\theta)=-d\Bigl(\frac{r_0}{r}\cos\theta_0\Bigr) = -d\Bigl(\sqrt{\frac{\Delta_0}{\Delta_\theta}}\cos\theta_0\Bigr) \\
    &= \Bigl(\sin\theta_0\sqrt{\frac{\Delta_0}{\Delta_\theta}} - \frac12 \cos\theta_0\frac{\Delta_0^{1/2}}{\Delta_\theta^{3/2}}\cdot 2\frac{\Lambda}{3}\bha^2\cos\theta_0\sin\theta_0\Bigr)d\theta_0
    = \sqrt{\frac{\Delta_0}{\Delta_\theta}}\Delta_\theta^{-1}\sin\theta_0\,d\theta_0.
\end{align*}
By Lemma~\ref{LemmaKO1}, we therefore have, for $V\in\cV(\pa M_s)$,
\begin{align}
  \ell_{\bhm,\bha}(V) &:= \int_{\pa M_s} V(\Err_{s,0})_{N T}\,|d h_s| \nonumber\\
\label{EqKOInt}
    &\ =2\Lambda\bhm_{\rm eff}\int_0^\infty\int_0^{2\pi} \int_0^\pi \chi'(t_0) V\bigl((\Delta_\theta-\tfrac13\Delta_0)d t-\bha\sin^2\theta_0\,d\phi\bigr)\sin\theta_0\,d\theta_0\,d\phi_0\,d t_0,
\end{align}
where $\bhm_{\rm eff}:=\bhm/\Delta_0^2$ is the effective mass of the triple $(p_0,\bhm,\bha)$. Particular conformal Killing vector fields $V$ on $(\pa M_s,h_s)$ include $\pa_t$ and $\pa_\phi$, and we compute
\[
  \ell_{\bhm,\bha}(\pa_t) = -\frac{16\pi\Lambda\bhm_{\rm eff}}{3},\quad
  \ell_{\bhm,\bha}(\pa_\phi) = \frac{16\pi\Lambda\bhm_{\rm eff}\bha}{3}.
\]
If $V$ is a rotation around an axis orthogonal to that corresponding to $\pa_\phi$, then the integrand in~\eqref{EqKOInt} vanishes pointwise, hence $\ell_{\bhm,\bha}(V)=0$ in this case. Passing to the boundary $\Sph^3$ of global de~Sitter space, with the KdS black hole sitting at the point $p_0=(1,0,0,0)\in\Sph^3$, we have $\pa_t=S_{C p_0}$ in the notation~\eqref{EqSConfKillSph} for some constant $C>0$ (only depending on $\Lambda$), while the rotations on the $\Sph^2$-factor of $\pa M_s$ which we considered above span the set $(\so_4)_{p_0}$ of rotations on $\Sph^3$ keeping $p_0$ fixed.

Consider rotations $V\in\so_4$ which are orthogonal to $(\so_4)_{p_0}$ with respect to the inner product $\la\cdot,\cdot\ra$ defined in~\eqref{EqKInnerProd}; the 3-dimensional space of such $V$ is spanned by rotation vector fields $R_j$, $j=2,3,4$, which are, say, 90 degree rotations in the planes determined by $p_0=(1,0,0,0)$ and $\hat e_2=(0,1,0,0)$, $\hat e_3=(0,0,1,0)$, $\hat e_4=(0,0,0,1)$, respectively, and which keep the orthogonal complement of $\mathspan\{p_0,\hat e_j\}$ in $\R^4$ fixed. But then the integrand in~\eqref{EqKOInt}, for each $j=2,3,4$, is odd either with respect to the reflection $\theta_0\mapsto\pi-\theta_0$ or with respect to the rotation $\phi_0\mapsto\phi_0+\pi$, hence $\ell_{\bhm,\bha}(R_j)=0$. A similar symmetry argument shows that $\ell_{\bhm,\bha}(S_{\hat e_j})=0$ for $j=2,3,4$.

Note now that $\pa_\phi$, written as a rotation matrix (rotating in the plane spanned by $\hat e_1$ and $\hat e_2$, while leaving the span of $p_0$ and $\hat e_3$ fixed), is given by
\[
  \pa_\phi = \begin{pmatrix} 0 & 0 & 0 & 0 \\ 0 & 0 & 1 & 0 \\ 0 & -1 & 0 & 0 \\ 0 & 0 & 0 & 0 \end{pmatrix}.
\]
By comparison with~\eqref{EqKParamMetric0}, we can thus summarize our calculations by
\begin{align*}
  \ell_{\bhm,\bha}(S_q) &= C_0\bhm_{\rm eff}\la p_0,q\ra,\quad q\in\R^4, \\
  \ell_{\bhm,\bha}(\fa) &= C_1\bhm_{\rm eff}\la\fa_0,\fa\ra,\quad \fa\in\so_4,
\end{align*}
where $C_0,C_1$ are nonzero real constants. We then have the following analogue of Lemma~\ref{LemmaSO2}:

\begin{lemma}
\label{LemmaKO2}
  Given admissible triples $b_1,\ldots,b_N$ as in Theorem~\usref{ThmKGlue}, $b_i=(p_i,\bhm_i,\fa_i)$, with the $p_i$ pairwise distinct, suppose $\chi_i\in\CI(\pa M)$ are cutoff functions, which are identically $1$ near $p_i$. Set $\Err_{N T}:=\sum_{i=1}^N(\Err_{b_i})_{N T}\in\CI(\Sph^3;T^*\Sph^3)$, where
  \[
    (\Err_{b_i})_{N T}(W) := \tau^{-4}P_0(\chi_i g_{p_i,\bhm_i,\fa_i} + (1-\chi_i)g_\dS\bigr)(\tau\pa_\tau,\tau W)|_{\tau=0},\quad W\in T\Sph^3.
  \]
  Then we have
  \[
    \int_{\Sph^3} V(\Err_{N T}) |d g_{\Sph^3}| = 0\quad\forall\,V\in\mathfrak{conf}_3
  \]
  if and only if $\{b_1,\ldots,b_N\}$ is balanced as in Definition~\usref{DefKBalance}.
\end{lemma}

The remainder of the gluing construction is very similar to the SdS gluing:

\begin{proof}[Proof of Theorem~\usref{ThmKGlue}]
  The only minor difference compared to the proof of Theorem~\ref{ThmSGlue} is the analogue of Proposition~\ref{PropSO}. Under the balance condition we can solve away the normal-tangential component of the error term using Delay's result. However, the $\tau^4$ leading order part of the error in general now has other nonvanishing components as well; but as demonstrated in Lemma~\ref{LemmaKO1}, these error terms lie in the range of $I(L_0,4)$ and can thus be solved away pointwise on $\pa M$ using a $\tau^4\CI(\pa M;S^2\,{}^0 T^*M)$ metric correction, with support of this correction contained in $\bigcup_i\supp d\chi_i$.
  
  The rest of the proof is the same: one constructs a formal solution in a generalized harmonic gauge as in Proposition~\ref{PropSGFormal}, solves away the remaining `trivial' error as in Proposition~\ref{PropSCNonlinear}, and thus obtains a solution of the Einstein vacuum equation by appealing to Lemma~\ref{LemmaSCUniq}.
\end{proof}

\bibliographystyle{alpha}

\begin{thebibliography}{DHR13}

\bibitem[AC05]{AndersonChruscielSimple}
Michael~T. Anderson and Piotr~T. Chru{\'s}ciel.
\newblock {A}symptotically simple solutions of the vacuum {E}instein equations
  in even dimensions.
\newblock {\em Communications in Mathematical Physics}, 260(3):557--577, 2005.

\bibitem[And05]{AndersonStabilityEvenDS}
Michael~T. Anderson.
\newblock Existence and stability of even-dimensional asymptotically de
  {S}itter spaces.
\newblock {\em Annales Henri Poincar\'e}, 6(5):801--820, 2005.

\bibitem[BL63]{BrillLindquist}
Dieter~R. Brill and Richard~W. Lindquist.
\newblock Interaction energy in geometrostatics.
\newblock {\em Physical Review}, 131(1):471, 1963.

\bibitem[Car68]{CarterHamiltonJacobiEinstein}
Brandon Carter.
\newblock {H}amilton--{J}acobi and {S}chr{\"o}dinger separable solutions of
  {E}instein's equations.
\newblock {\em Communications in Mathematical Physics}, 10(4):280--310, 1968.

\bibitem[CB52]{ChoquetBruhatLocalEinstein}
Yvonne Choquet-Bruhat.
\newblock Th{\'e}or{\`e}me d'existence pour certains syst{\`e}mes
  d'{\'e}quations aux d{\'e}riv{\'e}es partielles non lin{\'e}aires.
\newblock {\em Acta mathematica}, 88(1):141--225, 1952.

\bibitem[CBG69]{ChoquetBruhatGerochMGHD}
Yvonne Choquet-Bruhat and Robert Geroch.
\newblock {G}lobal aspects of the {C}auchy problem in general relativity.
\newblock {\em Communications in Mathematical Physics}, 14(4):329--335, 1969.

\bibitem[CD02]{ChruscielDelaySimple}
Piotr~T. Chru{\'s}ciel and Erwann Delay.
\newblock Existence of non-trivial, vacuum, asymptotically simple spacetimes.
\newblock {\em Classical and Quantum Gravity}, 19(9):L71, 2002.

\bibitem[CD03]{ChruscielDelayMapping}
Piotr~T. Chru{\'s}ciel and Erwann Delay.
\newblock On mapping properties of the general relativistic constraints
  operator in weighted function spaces, with applications.
\newblock {\em M\'em. Soc. Math. Fr. (N.S.)}, (94):vi+103, 2003.

\bibitem[CIP04]{ChruscielIsenbergPollackPRL}
Piotr~T. Chru{\'s}ciel, James Isenberg, and Daniel Pollack.
\newblock Gluing initial data sets for general relativity.
\newblock {\em Physical review letters}, 93(8):081101, 2004.

\bibitem[CIP05]{ChruscielIsenbergPollackEngineering}
Piotr~T. Chru{\'s}ciel, James Isenberg, and Daniel Pollack.
\newblock Initial data engineering.
\newblock {\em Communications in mathematical physics}, 257(1):29--42, 2005.

\bibitem[CM03]{ChruscielMazzeoManyBH}
Piotr~T. Chru{\'s}ciel and Rafe Mazzeo.
\newblock On 'many-black-hole' vacuum spacetimes.
\newblock {\em Classical and Quantum Gravity}, 20(4):729, 2003.

\bibitem[Cor00]{CorvinoScalar}
Justin Corvino.
\newblock Scalar curvature deformation and a gluing construction for the
  {E}instein constraint equations.
\newblock {\em Comm. Math. Phys.}, 214(1):137--189, 2000.

\bibitem[Cor13]{CortierKdSGluing}
Julien Cortier.
\newblock Gluing construction of initial data with {K}err--de {S}itter ends.
\newblock {\em Ann. Henri Poincar\'e}, 14(5):1109--1134, 2013.

\bibitem[CP08]{ChruscielPollackKottler}
Piotr~T. Chru{\'s}ciel and Daniel Pollack.
\newblock Singular {Y}amabe metrics and initial data with exactly
  {K}ottler--{S}chwarzschild--de {S}itter ends.
\newblock {\em Ann. Henri Poincar\'e}, 9(4):639--654, 2008.

\bibitem[CS06]{CorvinoSchoenAsymptotics}
Justin Corvino and Richard~M. Schoen.
\newblock On the asymptotics for the vacuum {E}instein constraint equations.
\newblock {\em J. Differential Geom.}, 73(2):185--217, 2006.

\bibitem[CS16]{CarlottoSchoenData}
Alessandro Carlotto and Richard Schoen.
\newblock {L}ocalizing solutions of the {E}instein constraint equations.
\newblock {\em Inventiones mathematicae}, 205(3):559--615, 2016.

\bibitem[Del12]{DelayCompact}
Erwann Delay.
\newblock {S}mooth compactly supported solutions of some underdetermined
  elliptic {PDE}, with gluing applications.
\newblock {\em Communications in Partial Differential Equations},
  37(10):1689--1716, 2012.

\bibitem[DeT82]{DeTurckPrescribedRicci}
Dennis~M. DeTurck.
\newblock Existence of metrics with prescribed {R}icci curvature: local theory.
\newblock {\em Invent. Math.}, 65(1):179--207, 1981/82.

\bibitem[DHR13]{DafermosHolzegelRodnianskiKerrBw}
Mihalis Dafermos, Gustav Holzegel, and Igor Rodnianski.
\newblock A scattering theory construction of dynamical vacuum black holes.
\newblock {\em Preprint, arXiv:1306.5364}, 2013.

\bibitem[FG85]{FeffermanGrahamAmbient}
Charles Fefferman and C.~Robin Graham.
\newblock Conformal invariants.
\newblock In {\em {\'E}lie Cartan et les math{\'e}matiques d'aujourd'hui -
  Lyon, 25-29 juin 1984}, number S131 in Ast{\'e}risque, pages 95--116.
  Soci{\'e}t{\'e} math{\'e}matique de France, 1985.

\bibitem[FG12]{FeffermanGrahamAmbientBook}
Charles Fefferman and C.~Robin Graham.
\newblock {\em The ambient metric}, volume 178 of {\em Annals of Mathematics
  Studies}.
\newblock Princeton University Press, Princeton, NJ, 2012.

\bibitem[Fri86a]{FriedrichDeSitterPastSimple}
Helmut Friedrich.
\newblock Existence and structure of past asymptotically simple solutions of
  {E}instein's field equations with positive cosmological constant.
\newblock {\em Journal of Geometry and Physics}, 3(1):101 -- 117, 1986.

\bibitem[Fri86b]{FriedrichStability}
Helmut Friedrich.
\newblock On the existence of {$n$}-geodesically complete or future complete
  solutions of {E}instein's field equations with smooth asymptotic structure.
\newblock {\em Comm. Math. Phys.}, 107(4):587--609, 1986.

\bibitem[Fri91]{FriedrichEinsteinMaxwellYangMills}
Helmut Friedrich.
\newblock {O}n the global existence and the asymptotic behavior of solutions to
  the {E}instein--{M}axwell--{Y}ang-{M}ills equations.
\newblock {\em Journal of Differential Geometry}, 34(2):275--345, 1991.

\bibitem[GL91]{GrahamLeeConformalEinstein}
C.~Robin Graham and John~M. Lee.
\newblock Einstein metrics with prescribed conformal infinity on the ball.
\newblock {\em Adv. Math.}, 87(2):186--225, 1991.

\bibitem[GZ03]{GrahamZworskiScattering}
C.~Robin Graham and Maciej Zworski.
\newblock Scattering matrix in conformal geometry.
\newblock {\em Inventiones mathematicae}, 152(1):89--118, 2003.

\bibitem[HV18]{HintzVasyKdSStability}
Peter Hintz and Andr{\'a}s Vasy.
\newblock {T}he global non-linear stability of the {K}err--de {S}itter family
  of black holes.
\newblock {\em Acta mathematica}, 220:1--206, 2018.

\bibitem[HZ18]{HintzZworskiHypObs}
Peter Hintz and Maciej Zworski.
\newblock Resonances for obstacles in hyperbolic space.
\newblock {\em Comm. Math. Phys.}, 359(2):699--731, 2018.

\bibitem[IMP02]{IsenbergMazzeoPollackWormholes}
James Isenberg, Rafe Mazzeo, and Daniel Pollack.
\newblock {G}luing and wormholes for the {E}instein constraint equations.
\newblock {\em Communications in Mathematical Physics}, 231(3):529--568, 2002.

\bibitem[IMP03]{IsenbergMazzeoPollackTopology}
James Isenberg, Rafe Mazzeo, and Daniel Pollack.
\newblock On the topology of vacuum spacetimes.
\newblock In {\em Annales Henri Poincar{\'e}}, volume~4, pages 369--383.
  Springer, 2003.

\bibitem[IMP05]{IsenbergMaxwellPollackGluing}
James Isenberg, David Maxwell, and Daniel Pollack.
\newblock A gluing construction for non-vacuum solutions of the
  einstein-constraint equations.
\newblock {\em Advances in Theoretical and Mathematical Physics},
  9(1):129--172, 2005.

\bibitem[KT93]{KastorTraschenManyBH}
David Kastor and Jennie Traschen.
\newblock {C}osmological multi-black-hole solutions.
\newblock {\em Physical Review D}, 47(12):5370, 1993.

\bibitem[Lin63]{LindquistIVP}
Richard~W. Lindquist.
\newblock {I}nitial-value problem on {E}instein--{R}osen manifolds.
\newblock {\em Journal of Mathematical Physics}, 4(7):938--950, 1963.

\bibitem[Luk12]{LukCharacteristic}
Jonathan Luk.
\newblock On the local existence for the characteristic initial value problem
  in general relativity.
\newblock {\em Int. Math. Res. Not.}, (20):4625--4678, 2012.

\bibitem[Maj47]{MajumdarSolution}
Sudhansu~D. Majumdar.
\newblock A class of exact solutions of {E}instein's field equations.
\newblock {\em Physical Review}, 72(5):390, 1947.

\bibitem[Maz91]{MazzeoEdge}
Rafe Mazzeo.
\newblock {E}lliptic theory of differential edge operators {I}.
\newblock {\em Communications in Partial Differential Equations},
  16(10):1615--1664, 1991.

\bibitem[Mel93]{MelroseAPS}
Richard~B. Melrose.
\newblock {\em The {A}tiyah--{P}atodi--{S}inger index theorem}, volume~4 of
  {\em Research Notes in Mathematics}.
\newblock A K Peters, Ltd., Wellesley, MA, 1993.

\bibitem[Mel96]{MelroseDiffOnMwc}
Richard~B. Melrose.
\newblock Differential analysis on manifolds with corners.
\newblock {\em Book, in preparation, available online}, 1996.

\bibitem[Mis63]{MisnerGeometrostatics}
Charles~W. Misner.
\newblock The method of images in geometrostatics.
\newblock {\em Annals of Physics}, 24:102--117, 1963.

\bibitem[MM87]{MazzeoMelroseHyp}
Rafe~R. Mazzeo and Richard~B. Melrose.
\newblock Meromorphic extension of the resolvent on complete spaces with
  asymptotically constant negative curvature.
\newblock {\em Journal of Functional Analysis}, 75(2):260--310, 1987.

\bibitem[P{\etalchar{+}}99]{PerlmutterEtAlLambda}
Saul Perlmutter et~al.
\newblock {M}easurements of {$\Omega$} and {$\Lambda$} from 42
  {H}igh-{R}edshift {S}upernovae.
\newblock {\em The Astrophysical Journal}, 517(2):565, 1999.

\bibitem[Pap45]{PapapetrouSolution}
A.~Papapetrou.
\newblock {A} {S}tatic {S}olution of the {E}quations of the {G}ravitational
  {F}ield for an {A}rbitary {C}harge-{D}istribution.
\newblock {\em Proceedings of the Royal Irish Academy. Section A: Mathematical
  and Physical Sciences}, 51:191--204, 1945.

\bibitem[Pen65]{PenroseAsymptotics}
Roger Penrose.
\newblock Zero rest-mass fields including gravitation: asymptotic behaviour.
\newblock In {\em Proceedings of the Royal Society of London A: Mathematical,
  Physical and Engineering Sciences}, volume 284, pages 159--203. The Royal
  Society, 1965.

\bibitem[R{\etalchar{+}}98]{RiessEtAlLambda}
Adam~G. Riess et~al.
\newblock {O}bservational {E}vidence from {S}upernovae for an {A}ccelerating
  {U}niverse and a {C}osmological {C}onstant.
\newblock {\em The Astronomical Journal}, 116(3):1009, 1998.

\bibitem[Ren90]{RendallCharacteristic}
Alan~D. Rendall.
\newblock {R}eduction of the {C}haracteristic {I}nitial {V}alue {P}roblem to
  the {C}auchy {P}roblem and {I}ts {A}pplications to the {E}instein
  {E}quations.
\newblock {\em Proceedings of the Royal Society of London. Series A,
  Mathematical and Physical Sciences}, 427(1872):221--239, 1990.

\bibitem[RSR18]{RodnianskiShlapentokhRothmanSelfSimilar}
Igor Rodnianski and Yakov Shlapentokh-Rothman.
\newblock {T}he asymptotically self-similar regime for the {E}instein vacuum
  equations.
\newblock {\em Geometric and Functional Analysis}, 28(3):755--878, 2018.

\bibitem[Sch08]{SchottenloherCFT}
Martin Schottenloher.
\newblock {\em A mathematical introduction to conformal field theory}, volume
  759 of {\em Lecture Notes in Physics}.
\newblock Springer-Verlag, Berlin, second edition, 2008.

\bibitem[Sch15]{SchlueCosmological}
Volker Schlue.
\newblock {G}lobal results for linear waves on expanding {K}err and
  {S}chwarzschild de {S}itter cosmologies.
\newblock {\em Communications in Mathematical Physics}, 334(2):977--1023, 2015.

\bibitem[Sch16]{SchlueWeylDecay}
Volker Schlue.
\newblock {D}ecay of the {W}eyl curvature in expanding black hole cosmologies.
\newblock {\em Preprint, arXiv:1610.04172}, 2016.

\bibitem[Sch19]{SchlueOpticaldS}
Volker Schlue.
\newblock {O}ptical functions in de {S}itter.
\newblock {\em Preprint, arXiv:1910.05799}, 2019.

\bibitem[Tay11]{TaylorPDE3}
Michael~E. Taylor.
\newblock {\em Partial differential equations {III}. {N}onlinear equations},
  volume 117 of {\em Applied Mathematical Sciences}.
\newblock Springer, New York, second edition, 2011.

\bibitem[Vas10]{VasyWaveOndS}
Andr{\'a}s Vasy.
\newblock The wave equation on asymptotically de {S}itter-like spaces.
\newblock {\em Advances in Mathematics}, 223(1):49--97, 2010.

\bibitem[Zwo16]{ZworskiRevisitVasy}
Maciej Zworski.
\newblock Resonances for asymptotically hyperbolic manifolds: {V}asy's method
  revisited.
\newblock {\em J. Spectr. Theory}, 2016(6):1087--1114, 2016.

\end{thebibliography}
\newcommand{\etalchar}[1]{$^{#1}$}

\end{document}